\documentclass[11pt]{amsart}

\usepackage{amsbsy, amsmath,amsfonts,amssymb,amsthm,amscd,amssymb,latexsym,}
\usepackage[latin1]{inputenc}
\usepackage[a4paper,lmargin=2cm,rmargin=1.5cm,tmargin=3.5cm,bmargin=3.5cm]{geometry}

\usepackage[dvips]{epsfig}
\usepackage{graphics}
\usepackage{psfrag}
 \usepackage{amsfonts,amsmath,amsthm}
\usepackage{amssymb}
\usepackage[all]{xy}
\usepackage{latexsym}
\usepackage{enumerate}
\usepackage{graphicx,color}
\usepackage{epsfig}
\usepackage{graphicx}

\usepackage{indentfirst}    

\usepackage{amsbsy, amsmath,amsfonts,amssymb,amsthm,amscd,amssymb,latexsym}
\usepackage{graphicx}
\usepackage[dvips]{epsfig}
\usepackage{graphics}
\usepackage{psfrag}
 
\usepackage{cancel}

\usepackage{enumerate}
\usepackage{indentfirst}
\usepackage{euscript}

 \theoremstyle{plain}
\newtheorem{proposition}{Proposition}

\newtheorem{remark}{Remark}

\newtheorem{theorem}{Theorem}

 \newcommand{\ep}{\varepsilon}

\usepackage[usenames,dvipsnames]{xcolor}

  \usepackage{color}

 \pagecolor{white}

 \newcommand{\R}{\mathbb R}

\newcommand{\Cc}{\mathcal C}

\newcommand{\Uu}{\mathcal U}

\newcommand{\Sa}{{\mathcal S}_\alpha}

\newcommand{\cv}{\cdot}

\newcommand{\fun}{\longrightarrow}

\begin{document}

 \title[ Axial Lines  at critical points of Surfaces in $\mathbb R^4$]{ Lines of axial curvature at critical points on  surfaces 
 mapped into 
  $\mathbb R^4$}

\author{ R. Garcia and J. Sotomayor }


\thanks{Both authors were partially supported by   CNPq. }


\begin{abstract}
In this paper are studied the simplest patterns of axial curvature lines (along which the normal curvature vector is at a vertex
of the ellipse of curvature)  near a critical point of a surface  mapped into $\mathbb R^4$.
These critical points, 
where  the rank of the mapping drops 
from 
$2$
to $1$, occur isolated   in  generic one parameter families of mappings of
surfaces into   $\mathbb R^4$.
As the parameter 
crosses 
a
critical bifurcation 
 value, at  which the mapping  has a critical point, 
 it is 
described 
how 
the axial umbilic points, which are the singularities of the axial curvature configurations at regular points, 
 move   along  smooth  arcs   to 
reach the critical point.
The numbers  of such arcs  and their  axial umbilic types, ( see \cite{sotogarcia1}, 
\cite{sotogarcia}),  
are fully described 
for 
a  typical family of mappings with a critical point.
\end{abstract}

\maketitle

\section{ Introduction}
\def\G{\Gamma}
\def\re{\text{\bf I\!\!  R}}
\def\s1{\text{\bf S}^1}
 \def\ene{\text{\bf I\!  N}}

 \def\ps{\psi}
\def\g{\gamma}

 \def\bs{\text{\bf S}}

 The study of the curvature of surfaces as a measure of how they bend  when mapped into  $\R ^k$,  for  $k  \geq   3, \, $ is the source of 
challenging problems  in Geometry and Analysis. 

Ideas and methods  coming from the Qualitative Theory of Differential Equations, Singularity Theory  and Dynamical Systems, such as Structural Stability,
 Critical Points  and Bifurcations have 
 been the subject of numerous recent research  contributions. The  departure point for this development, however, is not disjoint from the work of the pioneers such as Monge, Gauss and Darboux,
 to mention just a few.  We refer the reader to    Little  \cite{little},  Sotomayor \cite{sgresenha}, Porteous \cite{Porteous} and Garcia-Sotomayor\cite{sotogarcia},  among others,  for presentations of the several  branches of this field and  for references.

According to Whitney's Immersion Theorem  \cite{lt}, a generic map of a surface into $\mathbb R^4$  is an immersion (i.e. has rank $2$ everywhere) which is one-to-one except at a discrete set of  pairs of {\it double} points at which the images of the map have transverse crossings. 

This paper studies how the typical normal curvature  geometric singularities  and their associated axial foliations  reach the  simplest {\it critical}  points, at which 
the  rank of a  mapping 
of  a surface into $\mathbb R^4$  
   depending on one parameter
drops to $1$. This  happens when, moving with the parameter,  a pair of double points with transversal images come together at the critical point.

The  interest in the study of the interaction of  geometric foliations  defined by normal curvature properties with  structurally stable critical points  occurring in the supporting surfaces
mapped into Euclidean spaces 
appear  in recent works.
 See 
 \cite{coneBSM}, \cite{ggs}, \cite{oli}, \cite{tari},
 to mention just a few references.

 In this paper  is presented  in detail a case study  of  
  the situation where
  the critical point is not structurally stable  but  appears stably 
  and generically 
  in one parameter families of
 regular mappings.

 In other words, this paper deals with the bifurcations of the   Configurations of Extremal Curvature Deviation  (also called Axial Configurations)    at the singular  points on a surface  whose mapping  into  $\R ^4$,   changes  with a parameter. 
Such  configurations  appear  in the families of curves  
 of  {\em maximal}  (also called Axial Principal) and {\em minimal}  (Axial Mean)  normal curvature deviation from the {\em mean normal curvature}.  When the image is in a $3-$dimensional subspace, 
 they are, respectively,  the  principal and mean normal curvature configurations.
 
Most singular points occur along arcs  of  {\em axiumbilc points}, where the mapping has rank $2$ (i.e. it is regular)   and the maximal and minimal normal curvature deviations coincide,
 that is the {\em Ellipse of Normal Curvature}  is a circle. 
 
 The principal configurations around the Stable Axiumbilic points 
 were  described  in \cite{sotogarcia1}   and  around 
 their generic regular bifurcation  values  in \cite{flausino_arxiv}. 
 
 In this paper are studied the critical singular points, at which the rank of the mapping is $1$. Such points appear persistently (i.e. transversally) 
 in one parameter families of mappings.  
 This is done for a special class of mappings with a codimension $3$ critical point.

 The conclusions  of this work  are outlined below.
 
Sections \ref{sec:deax} and \ref{sec:3} 
 contain a review of the standard  theory of  the configurations of axial curvature lines and  
 axiumbilic singular points $E_3,\; E_4, \; E_5$
  appearing typically   
 at regular  points
of the  mapping, as established in \cite{sotogarcia1, sotogarcia}.

In section \ref{sec:4} is carried out the
study of a special typical family  $\alpha^a$, see equation \eqref{eq:familyW},
 reminiscent of the  Whitney Umbrella   critical point.
Depending on 
a cubic term,  represented by the parameter  $a$ in equation
 (\ref{eq:familyW}),
the {\it index}  of the axial
configuration  
 around the critical point
 can be    either  $1/2$  or $0$.
 
 The 
 axial configurations, for most parameters  $a$  are illustrated in Figure
 \ref{fig:acm} for index $1/2$ and  in Figure 
\ref{fig:aco} for index $0$.

As the deformation parameter
 crosses the critical bifurcation value, the following  holds:

In the index $1/2$  case two arcs of axiumbilic singularities of the
 $E-$types
 converge to, at crossing, and  emerge from, after crossing,   the critical point.
Two generic topological patterns, depending on the types $E_3$ and $E_4$ involved,  are possible. See Figure  \ref{fig:ace}.

In the index $0$ case four arcs of axiumbilc points, two of type $E_3$ 
and two of type $E_5$ 
converge toward the critical point and they are eliminated after crossing.
Two generic combinatorial arrangements are possible in this case. 
See Figure   \ref{fig:ace35}.

The authors believe that the results outlined above describe  
 the axial configurations at the generic  critical point of a surface mapped into $\R ^4$, as illustrated in Figures \ref{fig:acm} and \ref{fig:aco}.  They also present  a  rough description of   partial elements 
of the transversal,  codimension $1$,  bifurcations occurring by the elimination of the critical point.  
For the full description  one must be carry out  a delicate analysis of the breaking of the axiumbilic separatrix connections in Figures  \ref{fig:ace35} and   \ref{fig:ace}, 
 due to the presence of coefficients of  third order jet of the mapping omitted in the example treated here.

\section{\label{sec:deax} Differential Equation of Axial Lines near a Singular Point }

Consider a mapping $\alpha$ of class 
$C ^ r,  r \geq 5,$ 
 of
$\mathbb R^ 2 \times  \mathbb R $ into $\mathbb R^ 4 $, endowed with the Euclidean inner product $\langle \cdot , \cdot  \rangle$.
Take coordinates $u, v, \epsilon $ in the domain and
$x, y, z, w$ in the target. 
Assume that the origin is mapped into the origin by $\alpha$.

Write 
$\alpha_{\epsilon} = \alpha( u, v, \epsilon )$  and refer to it as a one-parameter  deformation of $\alpha_{0}$.

The 
{\it critical}
 set of $\alpha$ is the set   $\Sa$
 of points 
$(u,v,\epsilon)$
 such that $D\alpha_\epsilon (u,v) $ has rank less than $2$.
  
  The 
  set of regular 
  points,
   where $D\alpha_\epsilon$ has rank $ 2,$   
   will be 
   denoted by 
    $\mathcal R _\alpha$.

Suppose that $\alpha_{0}$ has rank $1$ at $(0,0)$, say  that  the 
chart $u,v$ is 
{\it adapted} 
if  $\frac{\partial \alpha_0}{\partial v} (0,0)  = 0$.

Critical points in this work will also  satisfy the {\it Whitney condition}. This means that 

\begin{equation} \label{cond:W}
 W=\alpha_u\wedge \alpha_{uv}\wedge \alpha_{vv} \ne 0.
 \end{equation}

\begin{remark}  \label{rem:WC}
Straight calculation shows that the Whitney condition defined by (\ref{cond:W}) at a critical point does not depend on the adapted chart.
\end{remark}

 The first fundamental form of the mapping $\alpha$ in the chart  $(u,v)$ is expressed by:

$$I_{\alpha} = \langle D\alpha, D\alpha \rangle  = E du^2 + 2Fdudv + G dv^2,$$

 where
$E=\langle \alpha_u,\alpha_u \rangle, F=\langle\alpha_u,\alpha_v\rangle$ and $G=\langle\alpha_v,\alpha_v\rangle.$

\begin{remark} \label{rem:WC-I}
It is a standard fact on positive definite quadratic forms that 
$D = EG-F^2 \geq 0$   and that $D = 0$  only  at  critical points.

Straight calculation shows that the   Whitney condition (\ref{cond:W})   is equivalent to require  that $D$  have  for $\alpha_0$  a non-degenerate critical point at $u=0, v=0$.
\end{remark}

Assuming the Whitney condition in an adapted chart, define the 
vectors

$\mathbf{N}_1=\alpha_u\wedge \alpha_{v}\wedge W$ and $\mathbf{N}_2=\alpha_u\wedge \alpha_{v}\wedge \mathbf{N}_1$, 
which 
give 
a normal frame at regular points of $\alpha$.

 The second  fundamental form  of $\alpha$, at regular points,  is
 defined 
  by:
$II_{\alpha} = II^1_{\alpha} \mathbf{N}_1/|\mathbf{N}_1| + II^2_{\alpha} \mathbf{N}_2/|\mathbf{N}_2|$, 
where
$II^i_{\alpha}, i=1,2$  is given by
$$II^i_{\alpha} := \langle D^2 \alpha, \mathbf{N}_i/|\mathbf{N}_i|\rangle  = e_i du^2+2f_i du dv+g_i dv^2$$
where,
$e_i=\langle\alpha_{uu},\mathbf{N}_i/|\mathbf{N}_i|\rangle, f_i=\langle\alpha_{uv},\mathbf{N}_i/|\mathbf{N}_i|\rangle$ and  $g_i=\langle\alpha_{vv},\mathbf{N}_i/|\mathbf{N}_i|\rangle$. 

The  \textit{mean normal curvature vector } of $\alpha$  is defined by $H=h_1 \mathbf{N}_1/|\mathbf{N}_1| +h_2 \mathbf{N}_2/|\mathbf{N}_2|$,
where 
$$h_i=\frac{Eg_i-2Ff_i+Ge_i}{2(EG-F^2)}.$$

For
  $\vec{v} \in  T_{p}   
  \mathbb R ^2 \setminus \{  0 \} $, the  \textit{normal curvature vector in the direction    $\vec{v}$}  
  at a regular point $p =(u,\vec{v}) $ of $\alpha$, 
  is defined by
\begin{equation}
\label{kn}
k_n=k_n(p,\vec{v}) = \frac{II_{\alpha}(\vec{v})}{I_{\alpha}(\vec{v})}= \frac{II^1_{\alpha}(\vec{v})}{I_{\alpha}(\vec{v})}\frac{\mathbf{N}_1}{|\mathbf{N}_1|} + \frac{II^2_{\alpha}(\vec{v})}{I_{\alpha}(\vec{v})}\frac{\mathbf{N}_2}{|\mathbf{N}_2|}
\end{equation}

At regular points $p$, 
the image of  $k_n$ restricted to the unitary circle   $S^1_{p}$  of $T_{p}
\mathbb R ^2$  (endowed with the metric  $I_{\alpha}$) describes
an ellipse $\ep_{\alpha}(p)$ centered   
at
  $H(p)$,  contained     in the space $N_{p}
$  normal  to $\alpha$.  It  is called  the    \textit{ellipse of curvature}
  of $\alpha$ at $p$.  
  See \cite{little},   \cite{sotogarcia}.
 
  Assuming that  $(e_1-g_1)f_2 - (e_2-g_2)f_1\neq 0$, $\ep_{\alpha}(p)$ is a standard ellipse or a circle, otherwise it can be a segment or a point.

 As  $k_n  |_{S^1_{p}}$ is quadratic, the pre-image of each point of the ellipse 
 consists 
 of two antipodal points in   $S^1_{p}$, and therefore   each point $\ep_{\alpha}({ p})$ is associated to a direction  in  $T_{p}
  \mathbb R ^2$. 
 Moreover, for each pair 
 of points in   $\ep_{\alpha}({ p})$, antipodally  symmetric 
 with respect 
 to  $H({  p})$, it is associated two orthogonal directions in     $T_{p} 
  \mathbb R ^2$,
   defining a pair of   \textit{lines} in  $T_p 
    \mathbb R ^2$,
 a {\it crossing}
  \cite{luisfernando}.

Regular  points 
where 
the ellipse  
$\ep_{\alpha}({ p})$
is a point or a circle 
are called  \textit{axiumbilic points } 
 of the 
 mapping 
 $\alpha$.  They will be denoted by $\Uu_{\alpha}$.
 
 For points $p$ away from $\Uu_{\alpha}$
 the directions in  
 $ T_{p} \mathbb R ^2 \setminus \{  0 \} $
 at which  $\|k_n (p, \cdot)-H(p)\|^2$ is 
 maximal 
define a pair of lines, a \emph{crossing}, 
called 
\emph{principal axial} or 
 of maximal normal curvature deviation.
 Analogously, directions at 
 which
  $\|k_n (p, \cdot)-H(p)\|^2$ is 
 minimal 
define a pair of lines, a \emph{crossing}, 
called 
\emph{mean axial} or 
 of minimal normal curvature deviation.
 
 The function  $\|k_n (p, \cdot) -H(p)\|^2$, 
 measures the deviation from the mean normal curvature at $p$.  
It is constant when $p$ is an axiumbilic point.

 \begin{remark} \label{rem:Princ-Mean-R3}
 The name axial directions come from the fact that there
 the normal curvature point at the extremes of the axes of the ellipse of curvature.
 
 The names \emph{principal} and \emph{mean} are justified by
 the case in which the mapping $\alpha$ has its image in 
  $\R ^3$ these lines are respectively  the principal directions 
  and the directions of the  mean normal three dimensional curvature.
 \end{remark}

  The set of axiumbilic points  $\Uu_{\alpha}$
  together with the critical points 
  $\Sa$ 
 of $\alpha$    
  are called the {\it singularities} of the fields of
  axial crossings   
   of the mapping $\alpha$.    By
   abuse of  terminology assignment  
   these fields of crossings are 
   sometimes referred to as line fields. 
   See ( \cite{sotogarcia1},\cite{sotogarcia}, \cite{luisfernando}).

The axial or  directions 
of extremal normal curvature deviation 
are defined by the equation

$$Jac(\|k_n-H\|^2, I_{\alpha})=0$$

which has four solutions for  $
p \notin \Uu_{\alpha}\cup {\mathcal S}_\alpha$,
the directions at which the normal curvature is at a vertex of $\ep_\alpha (p)$.
At  $p  \in {\Uu}_{\alpha}$ it vanishes identically  and at $ p \in \cup {\mathcal S}_\alpha$ it is not defined. Below will be shown how to extend it to critical points satisfying the Whitney condition in \ref{cond:W}. 

According to  \cite{sotogarcia1} and  \cite{sotogarcia}, the  \textit{differential equation of axial lines}  is given by
\begin{equation} \label{quartica.axiais}
\aligned
{\mathcal G}=& a_4 dv^4+a_3 dv^3 du+ a_2 dv^2 du^2 + a_1 dv du^3 + a_0 du^4 = 0,\\
a_1=&
4E^3(g_1^2+g_2^2)+4G(EG-4F^2)(e_1^2+e_2^2) \\
+&
32EFG(e_1f_1+e_2f_2)-16E^2G(f_1^2+f_2^2)-8E^2G(e_1g_1+e_2g_2),\\
 a_0=&
4F(EG-2F^2)(e_1^2+e_2^2)-4E(EG-4F^2)(e_1f_1+e_2f_2) \\
+&
-8E^2F(f_1^2+f_2^2)-4E^2F(e_1g_1+e_2g_2)+4E^3(f_1g_1+f_2g_2)\\
 Ea_2 &=-6G a_0+ 3F a_1;\\
E^2a_3 &= (4F^2-EG) a_1 - 8FG a_0;\\
 E^3a_4 &=G(EG-4F^2)a_0+ F(2F^2-EG)a_1.
 \endaligned
\end{equation}

Define the functions
 
\begin{equation}\label{eq:lmn} \aligned L_1=&Fg_1-Gf_1, \;\;\; M_1=Eg_1-Ge_1,\;\;\;  N_1=Ef_1-Fe_1\\
 L_2=&Fg_2-Gf_2, \;\;\; M_2=Eg_2-Ge_2,\;\;\;  N_2=Ef_2-Fe_2\endaligned\end{equation}
 
 It follows that:
 
 \begin{equation}\label{eq:ai}\aligned a_0=& 4(M_1N_1 +M_2N_2)E-8(N_1^2+N_2^2)F\\
 a_1=&4 (M_1^2+ M_2^2)E-16(N1^2+N_2^2)G\\
 a_2=& 12(M_1^2+M_2^2)F-24(M_1N_1+M_2N_2) G\\
 a_3=& 16(M_1L_1+M_2L_2)F-4(M_1^2+M_2^2+4N_1L_1+4N_2L_2)G\\
 a_4=& 8(L_1^2+L_2^2)F-4(M_1L_1+M_2L_2) G
 \endaligned
 \end{equation}

\begin{proposition}  
\label{eq.dif.axial}
Let  $\alpha: M \fun \R^4$ be a mapping of class  ${\Cc}^r, \ r\geq 5$, of a
smooth surface having critical  points satisfying the  Whitney condition \eqref{cond:W}. Denote the first fundamental form of $\alpha$    by:
$$I_{\alpha} = E du^2 + 2Fdudv + G dv^2$$
and the second fundamental form by:
$$II_{\alpha}= (e_1 du^2+2f_1 du dv+g_1 dv^2)\mathbf{N}_1/|\mathbf{N}_1| + (e_2 du^2+2f_2 du dv+g_2 dv^2) \mathbf{N}_2/|\mathbf{N}_2|$$
where $\mathbf{N}_1=\alpha_u\wedge \alpha_{v}\wedge W$ and $\mathbf{N}_2=\alpha_u\wedge \alpha_{v}\wedge \mathbf{N}_1$,   is a 
frame outside the critical points of $\alpha$  and  
such that $\{\alpha_u, \alpha_v, \mathbf{N}_1/|\mathbf{N}_1|,\mathbf{N}_2/|\mathbf{N}_2|\}$ is a positive frame at regular points of $\alpha$.

\begin{enumerate}
\item[i)] The differential equation

\begin{equation}\label{eq:dasing}\aligned
\overline{{\mathcal G}}=&\overline{a_4 } dv^4+\overline{a_3 } E dv^3 du+ \overline{a_2 }E^2 dv^2 du^2 + \overline{a_1 } E^3dv du^3 + \overline{a_0 } E^3 du^4 = 0\\
\overline{a_0}=&4E[(EG-F^2)\overline{N_1}\; \overline{M_1}+\overline{M_2}\;\overline{N_2}] -8F[(EG-F^2)\overline{N_1}^2+\overline{N_2}^2]\\
\overline{a_1}=& 4E[(EG-F^2)\overline{M_1}^2+\overline{M_2}^2]-16G[(EG-F^2)\overline{N_1}^2+\overline{N_2}^2] \\
  E\overline{a_2} &=-6G \overline{a_0} + 3F \overline{a_1} ;\\
E^2\overline{a_3}  &= (4F^2-EG) \overline{a_1}  - 8FG \overline{a_0} ;\\
 E^3\overline{a_4}  &=G(EG-4F^2)\overline{a_0} + F(2F^2-EG)\overline{a_1}.
\endaligned
\end{equation}
is a regular extension of the differential equation of axial lines to the singular set of
$\alpha$. Here the functions $\overline{M_i}$ and $\overline{N_i}$ are given in equation \eqref{eq:lmnb}.

\item [ii)] The axiumbilic and singular points of  $\alpha$  are characterized by  $\overline{a_0}=\overline{a_1}=0$.
\end{enumerate}

\end{proposition}

\begin{proof}
To obtain a regular extension of the differential
equation 
 of axial  curvature lines to the set of 
 critical 
 points $\Sa$ it is convenient to write 
 
 $\overline{e_i}= |\mathbf{N}_i|e_i=\langle \alpha_{uu},\mathbf{N}_i\rangle,\;\overline{f_i}= |\mathbf{N}_i|f_i=\langle \alpha_{uv},\mathbf{N}_i\rangle,$ \;
$\overline{g_i}= |\mathbf{N}_i|g_i=\langle \alpha_{vv},\mathbf{N}_i\rangle. $

At a critical 
 point $p\in \Sa$ it holds
 that $\mathbf{N}_1(p)=\mathbf{N}_2(p)=0$ 
 only at $p=(0,0)$
 in an adapted frame
 with the Whitney condition imposed.
 
From differential equation \eqref{quartica.axiais}  it follows that:

\begin{equation}\label{eq:a1aob}
 \aligned  a_0 |\mathbf{N}_1|^2|\mathbf{N}_2|^2
=&4F(EG-2F^2)(\overline{e_1}^2|\mathbf{N}_2|^2 + \overline{e_2}^2 |\mathbf{N}_1|^2)\\
-&  4E(EG-4F^2)(\overline{e_1}\cv \overline{f_1}|\mathbf{N}_2|^2 +\overline{e_2} \cv \overline{f_2}|\mathbf{N}_1|^2) \\
 +& 4E^3(\overline{f_1}  \cv\overline{g_1} |\mathbf{N}_2|^2 +  \overline{f_2} \cv\overline{g_2}|\mathbf{N}_1|^2)
-  4E^2F (\overline{e_1}\cv \overline{g_1}|\mathbf{N}_2|^2+\overline{e_2} \cv \overline{g_2}|\mathbf{N}_1|^2) \\
-& 8E^2F (\overline{f_1} ^2 |\mathbf{N}_2|^2+  \overline{f_2}^2|\mathbf{N}_1|^2)\\
  a_1|\mathbf{N}_1|^2|\mathbf{N}_2|^2 =& 4G(EG-4F^2)(\overline{e_1}^2 |\mathbf{N}_2|^2
  + \overline{e_2}^2  |\mathbf{N}_1|^2)\\
+ & 32 EFG(\overline{e_1} \cv \overline{f_1} |\mathbf{N}_2|^2+ \overline{e_2}\cv \overline{f_2} |\mathbf{N}_1|^2)
 + 4E^3(\overline{g_1}^2  |\mathbf{N}_2|^2+ \overline{g_2}^2|\mathbf{N}_1|^2)\\
- & 8E^2G (\overline{e_1}\cv \overline{g_1} |\mathbf{N}_2|^2+\overline{e_2}\cv \overline{g_2}|\mathbf{N}_1|^2)
-  16E^2G (\overline{f_1}^2 |\mathbf{N}_2|^2 + \overline{f_2}^2|\mathbf{N}_1|^2) 
\endaligned
\end{equation}

By definition of the 
normal frame $\{\mathbf{N}_1,\mathbf{N}_2\}$, under the Whitney condition,  it follows that $|\mathbf{N}_2(p)|^2= (EG-F^2) |\mathbf{N}_1(p)|^2$.

As above define the functions
\begin{equation}\label{eq:lmnb} \aligned \overline{L_1}=&F\overline{g_1}-G\overline{f_1}, \;\;\; \overline{M_1}=E\overline{g_1}-G\overline{e_1},\;\;\;  \overline{N_1}=E\overline{f_1}-F\overline{e_1}\\
\overline{L_2}=&F\overline{g_2}-G\overline{f_2}, \;\;\; \overline{M_2}=E\overline{g_2}-G\overline{e_2},\;\;\; \overline{N_2}=E\overline{f_2}-F\overline{e_2}\endaligned\end{equation}

 Therefore, defining 
 $$\aligned \overline{a_0}= & a_0 |\mathbf{N}_2|^2, \;\;\;\;\; \overline{a_1}=   a_1| |\mathbf{N}_2|^2,\\
  E\overline{a_2} &=-6G \overline{a_0} + 3F \overline{a_1} ;\\
E^2\overline{a_3}  &= (4F^2-EG) \overline{a_1}  - 8FG \overline{a_0} ;\\
 E^3\overline{a_4}  &=G(EG-4F^2)\overline{a_0} + F(2F^2-EG)\overline{a_1}.
\endaligned$$
and performing the simplifications  using \eqref{eq:lmnb}, the result stated is obtained.\end{proof}

\section{\label{sec:3}  Axial configurations near axiumbilic points on surfaces of $\mathbb R^4$ }

In this section  will be 
recalled  
  the qualitative behavior of the axial configurations  
around   
   a neighborhood of an axiumbilic point:
 principal, corresponding  to the maxima and minimal normal curvature deviation from $H_\alpha$.  
 The first is  denoted by   $\mathbb P_\alpha$,
  and
 the second by    $\mathbb Q_\alpha$.

\begin{proposition} (\cite{sotogarcia1}, \cite{sotogarcia})
\label{prop.axial} 
Let $p$ be an axiumbilic point. Then there exists a parametrization $(x,y,R(x,y),S(x,y))$ and a homotety in   $\R^4$ such that the 
differential equation of axial lines is given by:
\begin{equation}
\label{axial.monge}
y(dy^4-6dx^2dy^2+dx^4)+(a x+b y)dxdy(dx^2-dy^2) + H(x,y,dx,dy)=0
\end{equation}
where  $H$ contains terms of order  
greater than 
or equal to  $2$ in  $(x,y)$.
 Moreover, the axiumbilic point  $p$ is transversal, if and only if,  $a \neq 0$. Here the coefficients $a$ and $b$ 
are calculated in terms of $j^3R(0)$ and $j^3S(0)$.
\end{proposition}

Let \begin{equation}\label{eq:41}
\aligned \Delta(a,b) &= (a+1)^2[I^3-27J^2],\qquad \text{where},\\
 I &=2a(\frac{a}{24}+1)+4+\frac{b^2}{4}\;\;\;\text{and}
  \;\; J = -\frac{2a}{3}[(\frac{a}{6}+1)(1-\frac{a}{24})+\frac{b^2}{16}].
\endaligned \end{equation}

\begin{theorem}
\label{th:41}
  Consider a transversal axiumbilic point, for which  $ a \ne 0$.   Then 
  in the notation of Proposition \ref{prop.axial},  the following holds:
\begin{itemize}
\item[i)]  If $\;\;\Delta(a,b)  < 0,\;\;\;  $ then the axial configurations $\mathbb P_\alpha$ and $\mathbb Q_\alpha$ are of type $E_3$,   with   three axiumbilic separatrices, as shown in Fig. \ref{fig:a345}, top.

\item[ii)] If  $\;\;   \Delta(a,b)  < 0$ and $ a  < 0$, with 
$a\neq -1$, 
 then the axial configurations $\mathbb P_\alpha$ and $\mathbb Q_\alpha$ are of type  $E_4$,
  with   four axiumbilic separatrices and  one parabolic sector, as shown in Fig. \ref{fig:a345}, center.

\item[iii)]  If $\;\; \Delta(a,b)<0$,   $\;\; a > 0 $, then the axial configurations $\mathbb P_\alpha$ and $\mathbb Q_\alpha$ are of type $E_5$, with five axiumbilic separatrices, as shown in Fig. \ref{fig:a345}, bottom.

  \end{itemize}

 \begin{figure}[ht]
 \psfrag{v}{v} \psfrag{N}{N}
   \psfrag{Nv}{{$N(v)$}}  \psfrag{tpm}{{$T_p{\mathbb M}$} }
 \psfrag{npm}{$N_p{\mathbb M}$}
 \psfrag{n1}{$\mathbf{N}_1$} \psfrag{n2}{$\mathbf{N}_2$}
\begin{center}
\includegraphics[scale=0.4]{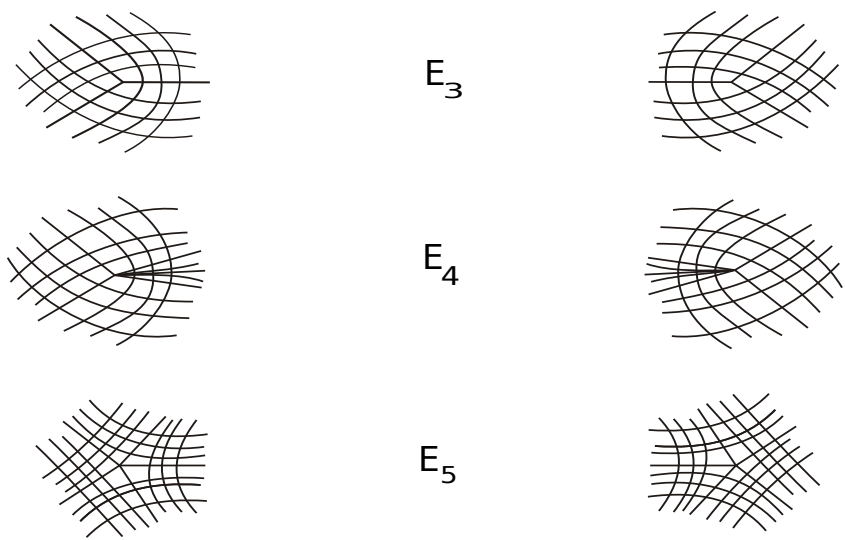}
\caption{ Axial Configurations near points $E_3$, $E_4$ and $E_5$ \label{fig:a345}}
\end{center}
\end{figure}
\end{theorem}

Figures  \ref{fig:a345} and \ref{fig:d345} illustrate Theorem \ref{th:41}.

 \begin{figure}[ht]
\begin{center}
\includegraphics[scale=0.20]{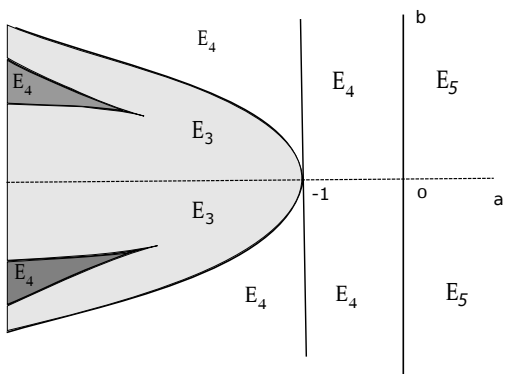}
\caption{ Axiumbilic types  $E_3$, $E_4$ and $E_5$ partitioning the plane $(a, b).$\label{fig:d345} } 
\end{center}
\end{figure}

 \section{  \label{sec:4}  Axial Configurations  near a Critical Point for 
  a Special Family of Mappings} 
 
 In this section the axial configuration of the special family of mappings defined by equation \eqref{eq:familyW} near the critical point will be established.  
 
   Consider the mapping 
\begin{equation}\label{eq:familyW} 
 \alpha^a(u,v)=\alpha(u,v)=(u,uv,v^2,\frac 16 av^3) 
 \end{equation}
 \noindent  and the vector $W=\alpha_u\wedge \alpha_{uv}\wedge \alpha_{vv}$.
  
Define the normal vectors
\; $\mathbf{N}_1=\alpha_u\wedge \alpha_{v}\wedge W$ and $\mathbf{N}_2=\alpha_u\wedge \alpha_{v}\wedge \mathbf{N}_1$.

The first fundamental form of $\alpha$ is given by:

\begin{equation}\label{eq:EFG} \aligned E=& 1+u^2,\;\;\;\;\;\;
F=   uv, \;\;\;\;\;\;
G =  u^2+4v^2+\frac 12 a^2 v^4\endaligned\end{equation}

The normal vectors are given by:

$$\aligned \mathbf{N}_1=&(-\frac 12 a^2v^4-4v^2, 4v+\frac 12 a^2v^3, -2u, -auv)\\
\mathbf{N}_2=&( -a u v^3, a u v^2, -a u^2v-2av^3-\frac 14 a^3v^5-\frac 14 a^3v^7-2a v^5,\\ &2u^2+8v^2+a^2v^4+a^2 v^6+8v^4).
\endaligned$$

The coefficients $\overline{e_1}=\langle \alpha_{uu}, \mathbf{N}_1\rangle$, $\overline{f_1}=\langle \alpha_{uv}, \mathbf{N}_1\rangle$, $\overline{g_1}=\langle \alpha_{vv}, \mathbf{N}_1\rangle$,
 $\overline{e_2}=\langle \alpha_{uu}, \mathbf{N}_2\rangle$, $\overline{f_2}=\langle \alpha_{uv}, \mathbf{N}_2\rangle$, $\overline{g_2}=\langle \alpha_{vv}, \mathbf{N}_2\rangle$ are given by:

 \begin{equation}\label{eq:efg12b} \aligned \overline{e_1}=& 0,\;\;\;
 \overline{f_1}= 4v+\frac 12 a^2 v^3,\;\;\;
\overline{g_1}= -u(4+a^2v^2)\\
 \overline{e_2}=&0,\;\;\;
\overline{f_2}= auv^2,\;\;\;\;\;\;\;\;\;\;\;\;\;\;
\overline{g_2}= \frac 12 av^3(1+v^2)(8+a^2v^2)
 \endaligned\end{equation}

From equations \eqref{eq:EFG} and \eqref{eq:efg12b} it follows that   the functions $\overline{a_0}$
and $\overline{a_1}$  in 
 equation \eqref{eq:dasing},  after multiplication by $-\frac{8}{(1+v^2)|\mathbf{N}_1|^2}$,  are given by:

 $$\aligned
  \overline{a_1}=&[(24-2a^2)v^2-8] u^2+\frac 12  v^4[(a^2-2a)v^2 +16-2a][(a^2+2a)v^2+16+2a]
 \\
\overline{ a_0}=&uv(8+24v^2+a^2v^2+2a^2v^4).
  \endaligned $$

 Therefore the  differential equation of axial lines in the singular surface is written 
 
  \begin{equation}\label{eq:23sing}
\aligned \overline{ {\mathcal G}}(u,v,du,dv)=& [ \overline{a_0}G(EG-4F^2)+ \overline{a_1}F(2F^2-EG)]dv^4 \\
 +& [-8\overline{a_0} EFG+ \overline{a_1} E(4F^2-EG)]dv^3du +\\
  [&-6\overline{a_0} GE^2+ 3\overline{a_1} FE^2]dv^2 du^2+ \overline{a_1}  E^3 dv du^3 + \overline{a_0} E^3 du^4  =0,  \\
   \overline{a_1}=&[(24-2a^2)v^2-8] u^2+\frac 12  v^4[(a^2-2a)v^2 +16-2a][(a^2+2a)v^2+16+2a]
 \\
\overline{ a_0}=&uv(8+24v^2+a^2v^2+2a^2v^4), \;\; E=1+u^2,\; F=uv, \;
G =  u^2+4v^2+\frac 12 a^2 v^4\endaligned
 \end{equation}

 \begin{proposition} \label{prop:lcsurface}  Consider the
mapping 
given by equation \eqref{eq:familyW}
 \noindent  in the plane $M = \R^2$ 
 and  its 
  Lie-Cartan variety $\overline{\mathcal G}(u,v,du,dv)=0$ defined by equation \eqref{eq:23sing}. 
 The projection $\pi:PM\to M$ restricted to $\overline{\mathcal G}^{-1}(0)$ 
  is a regular 
  four-fold covering 
  outside the projective line $\pi^{-1}(0)$.

Outside a neighborhood of the vertical direction $(0,[0:1])$, where 
 the variety $\overline{{\mathcal G}}^{-1}(0)$ 
 has a degenerate
  singular point, 
  it  is the union of two regular surfaces intersecting transversally along the projective line.

\end{proposition}

\begin{proof}
As the singular point is of Whitney type it follows that 
it
is isolated and in a punctured neighborhood
$U_0=U\setminus\{0\}$ of $0$ the map 
$\alpha$
is an immersion.
Below it is shown that
there are no axiumbilic points of $\alpha$ in $U_0$.

From equation \eqref{eq:23sing} it follows that
 $$\aligned \overline{a_1}(u,v)=&[(24-2a^2)v^2-8] u^2+\frac 12  v^4[(a^2-2a)v^2 +16-2a][(a^2+2a)v^2+16+2a]
 \\
\overline{ a_0}(u,v)=&uv(8+24v^2+a^2v^2+2a^2v^4)
 \endaligned $$

 Therefore, $  \overline{ a_0}(u,v)=0$ if only if $u=0$ or $v=0$.
  As $\overline{ a_1}(u,0)=-8u^2$ and $\overline{ a_1}(0,v)=
  \frac 12  v^4[(a^2-2a)v^2 +16-2a][(a^2+2a)v^2+16+2a]$ it follows that
  there are no axiumbilic in a punctured neighborhood of $0$.
 \end{proof} 
 
Outside a neighborhood of $(0,[0:1])$ the variety  $\overline{\mathcal G}=0$  is the union of two regular surfaces which  intersect transversally along the projective axis.  
  In Fig. \ref{fig:cm} is sketched the topological type of $\overline{\mathcal G}=0$ near the critical point  $(0,[0:1])$  with a  cut  along the projective axis when $|a|<8$. 
 
 In Fig. \ref{fig:4discos} is shown the topological type of $\overline{\mathcal G}=0$  near the critical point $(0,[0:1])$  when $|a|>8$.

 \begin{figure}[ht]
\begin{center}
\fbox{\includegraphics[scale=0.16]{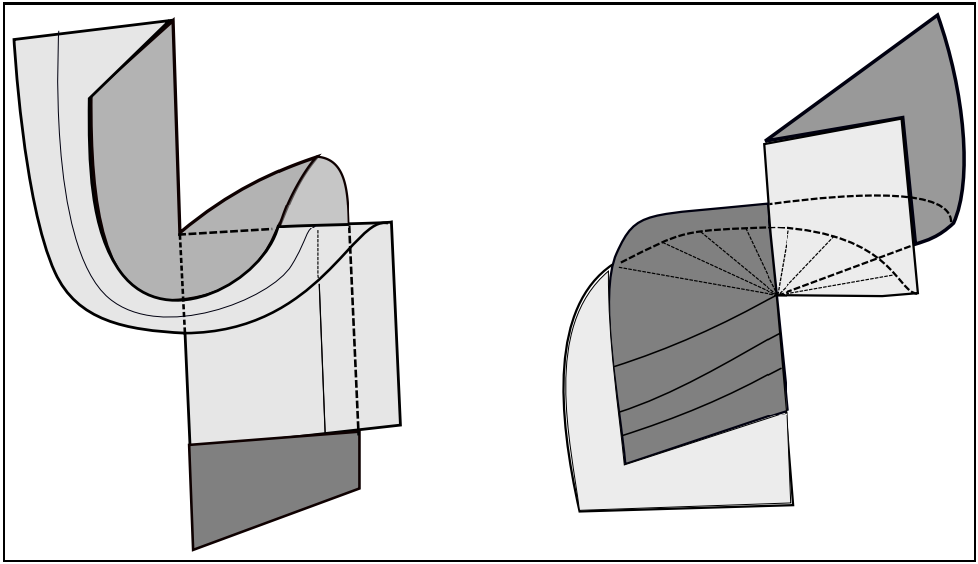}} 
\caption{ The two sheets  of the Lie-Cartan surface $\overline{\mathcal G}=0$ over the critical point of  the map $\alpha(u,v)=(u,uv,v^2, a/6v^3)$  for  $|a|<8$,  union of two topological cylinders near the critical point $q=0$. 
 Gluing the two pictures (left and right) by juxtaposition  along the vertical $q-$axis is recovered the singular surface consisting on  two  crossing topological cylinders, locally one on the  plane  of the drawing and the other  transversal to it.
\label{fig:cm} }
\end{center}
\end{figure}

 \begin{figure}[ht]
\begin{center}
\includegraphics[scale=1.20]{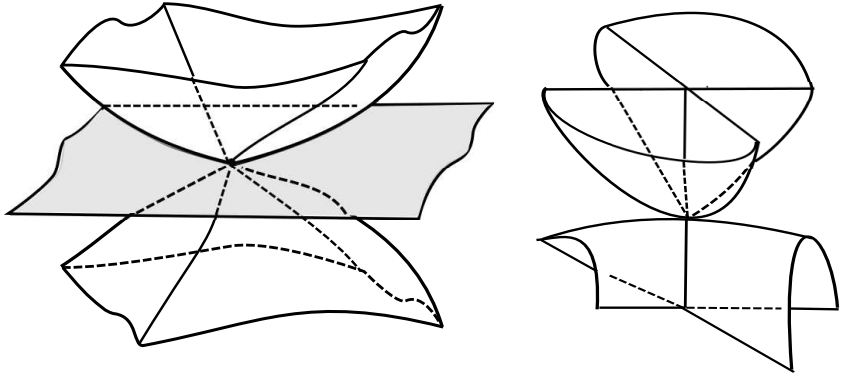}
\caption{ Singular   Lie-Cartan surface  $\overline{\mathcal G}=0$ of the map $\alpha(u,v)=(u,uv,v^2, a/6v^3)$. For $|a|>8$ is the  union of four topological punctured disks. Three of them are near the singular point, the other has the projective line in its closure.\label{fig:4discos}  }
\end{center}
\end{figure}

\begin{proposition}\label{prop:bh}
Consider the planar blowing-up $\psi(u, t) = (u, tu)$ around the origin.
Then, in $(u, t)$-coordinates, Equation \eqref{eq:23sing} restricted to the a small neighborhood of  t-axis has the form:
$$ du^3[8dt+ut^3(2t^2+1)(a^2+16)du]+0(u^2)]=0$$
Therefore, the pull-back of the axial configuration restricted to a small neighborhood
of the t-axis
 is as in Figure \ref{fig:bh}.
Three axial lines are almost vertical 
e one is transversal to axis $t$ and almost horizontal.

\begin{figure}[ht]
\begin{center}
\fbox{
\includegraphics[scale=0.55]{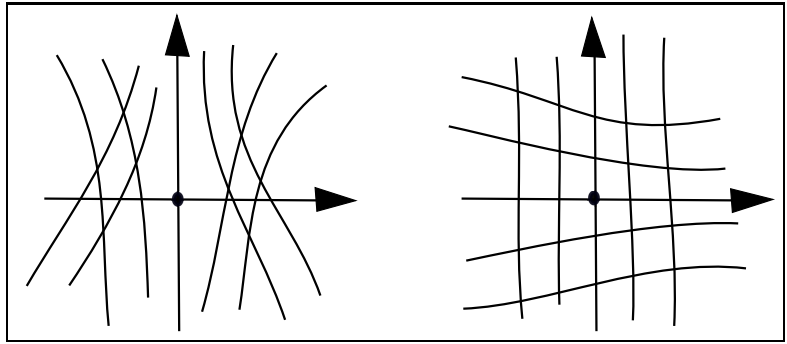}}
\caption{Pull-back of the axial configuration restricted to a small neighborhood
of the t-axis.\label{fig:bh} }
\end{center}
\end{figure}
\end{proposition}

\begin{proof} The differential equation \eqref{eq:23sing} 
restricted
 to $v=0$ is given by:
$-16u^2 dudv (du-udv)(du+u dv)=0$. 
One 
axial direction is horizontal $dv=0$, 
one is vertical $du=0$ 
and two 
are 
 almost vertical $\frac{du}{dv}=u$ and 
 $\frac{du}{dv}=-u$.
 
The proof follows from direct calculations 
leading  to 
$\psi_*(\overline{{\mathcal G}} )$ as stated. 
\end{proof}

 \begin{proposition}\label{prop:bv2}
Consider the planar blowing-up $\vartheta(u,v)=(r^2\sin \theta, r\cos \theta)$ around the origin.
Then, in $(\theta, r)$-coordinates, Equation \eqref{eq:23sing} restricted to the a small neighborhood of the 
$\theta-$axis  in the region $\{r\geq 0\}$ 
has the form:

\begin{equation}\label{eq:bv2}
\aligned
\vartheta_{*}(\overline{{\mathcal G}}  )
 =&dr^3 .[  2\cos \theta \sin \theta((20-a^2)\cos^4\theta\sin^2\theta+4\sin^4\theta+(a^2-56)\cos^8\theta) dr\\
  +& r[  (7a^2-384)\cos^{10}\theta 
 + (260-a^2)\cos^8\theta+(-7a^2+32)\cos^6\theta+(24+2a^2)\cos^4\theta\\
 -&4\cos^2\theta+8]d\theta  ]+ r^2.O(r,\theta,dr,d\theta)
  ={\mathcal E}_r(\theta,r,d\theta,dr)  
 \endaligned
 \end{equation}
 
 The singular points of ${\mathcal E}_r(\theta,r,d\theta,dr)$ 
 in the 
 $\theta$-axis are given by
 $$ 2\cos \theta\sin \theta((20-a^2)\cos^4\theta\sin^2t+4\sin^4\theta+(a^2-56)\cos^8\theta) =0.$$
 Near the $\theta$-axis three axial 
line fields  
 are almost
  horizontal and the other is defined by the 
  following  differential equation:
 
 \begin{equation} \label{eq:sing}  
 {  \aligned \omega= &  2\cos \theta \sin \theta((20-a^2)\cos^4\theta\sin^2\theta+4\sin^4\theta+(a^2-56)\cos^8\theta+r.O(r)) dr\\
 +& r[  (7a^2-384)\cos^{10}\theta+(260-a^2)\cos^8\theta+(-7a^2+32)\cos^6\theta+(24+2a^2)\cos^4\theta\\
 -&4\cos^2\theta+8+r.O(r)]d\theta
 =0 \endaligned}
 \end{equation}

 \end{proposition}

 \begin{proof} Direct calculations shows that
 $  \vartheta_{*}({\mathcal E} )$ is as stated. \end{proof}

  \begin{proposition} \label{prop:anasing} 
 The differential equation $\omega=0$ given by equation \eqref{eq:sing}
 assuming 
 that 
 $|a| \ne \sqrt{56}$ and $|a|\ne 8$ 
 has  either eight or twelve singular points in the interval $[0,2\pi)$ with
 three or five hyperbolic singular points in the interval
  $(-\frac{\pi}2,\frac{\pi}2) $ contained in the $\theta$-axis.  See Fig. \ref{fig:meio1} and Fig. \ref{fig:zero}.

Moreover,

\begin{enumerate}
\item[i)]
If $|a|< \sqrt{56}$  six  singular points are hyperbolic saddles and  $(\pm \frac{\pi}2,0)$ are hyperbolic nodes.

\item[ii)] If  $  \sqrt{56} <|a|< 8$,  eight 
singular points are hyperbolic saddles and 
the 
 point s$(\theta,r)=(0,0)$ and $(\pi,0)$ and $(\pm \frac{\pi}2,0)$   are  hyperbolic nodes.

\item[iii)] If  $8<|a|$,   ten singular points  are hyperbolic saddles and $(\pm \frac{\pi}2,0)$ are hyperbolic nodes.
\end{enumerate}
\end{proposition}

\begin{proof}
The singular points of $\omega=0$ are given  by
$$r=0, \;\;\; \cos \theta \sin \theta[(20-a^2)\cos^4\theta\sin^2\theta+4\sin^4\theta+(a^2-56)\cos^8\theta]=0.$$

Writing this equation using the  relations 
 $t=\tan\theta, \;\ \cos\theta=\frac{1}{\sqrt{1+t^2}},\;\  \sin\theta=\frac{t}{\sqrt{1+t^2}}$
  it follows that it  is equivalent to:

$$\aligned p(t)=& t[ 4t^8+8t^6-(a^2-24)t^4-(a^2-20)t^2+a^2-56],\\
  t=& \tan \theta \;\;\;\text{ and} \;\;\; \theta=\pm \frac{\pi}2.\endaligned$$
 
The polynomial $p(t)$ has the following factorization:
 
$$\aligned p(t)=&4tp{_1} (t^2)p{_2} (t^2),\\
           p{_1}(t) =& t^2 + t + \frac{5}{2} - \frac{1}{8}\sqrt[2]{a^4 - 56a^2 + 1296},\\ 
       p{_2}(t) =& t^2 + t + \frac{5}{2} + \frac{1}{8}\sqrt[2]{a^4 - 56a^2 + 1296}.
       \endaligned$$

The polynomial $p_1$ always has two   real simple roots, one positive and the other negative.

The polynomial $p_2$ has a positive root for $|a|>\sqrt{56}$ 
and for $|a|=\sqrt{56}$, $p_2(0)=0$ and for $|a|<\sqrt{56}$ the roots of $p_2$ are negative or complex.

So it follows that:

\begin{enumerate}
\item[i)] For $|a|<\sqrt{56}$, $\omega=0$ has three singular points in the interval $(-\frac{\pi}2,\frac{\pi}2).$
\item[ii)] For $|a|>\sqrt{56}$ and $|a|\ne 8$, $\omega=0$ has five singular points in the interval $(-\frac{\pi}2,\frac{\pi}2).$
\end{enumerate}

The differential equation $\omega=0$  
has the same solution curves as
the vector field $X=\mathbf{P}\frac{\partial}{\partial\theta}+\mathbf{Q} \frac{\partial}{\partial\theta}$, where

\begin{equation}\label{eq:campoX}
\aligned \mathbf{P}(\theta,r)= &2\cos \theta \sin \theta[(20-a^2)\cos^4\theta\sin^2\theta+4\sin^4\theta+(a^2-56)\cos^8\theta+r.O(r)] =P(\theta)+r.O(r) \\
\mathbf{Q}(\theta,r)=& -r[  (7a^2-384)\cos^{10}\theta+(260-a^2)\cos^8\theta+(-7a^2+32)\cos^6\theta+(24+2a^2)\cos^4\theta\\
-&4\cos^2\theta+8+r.O(r)] 
 =-r[Q(\theta)+r.O(r)]
\endaligned
\end{equation}

The jacobian of $DX(\theta,0)$  at the singular point  $(\theta,0)$, $t=\tan\theta$, is given by: 
$- P^\prime(\theta)Q(\theta)=\frac{1}{(1+t^2)^{10}}r_a t_a$,
where

$$ \aligned r_a=& -[8t^{10} +36t^8+ (2a^2+88)t^6+(160-a^2)t^4+(420-9a^2)t^2+a^2-64]\\
t_a=&2[-4t^{10}+12t^8+(5a^2-64)t^6+(2a^2-20)t^4+(564-12a^2)t^2+a^2-56]\endaligned$$

At $\theta=\pm \frac{\pi}2$ the jacobian of $DX(\pm\frac{\pi}2,0)$ is always equal to $64$
and at $\theta=0$ is given by:$ -2(a^2-56)(a^2-64).$

Evaluation of the resultants of polynomials, abbreviated by {\it res},
 give 
that

$$\aligned \text{res}(p(t),r_a,t)=&-1073741824(a^2-64)(5a^2-243)^2(a^2-36)^{12}\\
 \text{res}(p(t),t_a,t)=& 549755813888 (a^2-56)^5(5 a^2-243)^2 (1296-56 a^2+a^4)^4
 \endaligned$$

The values $a^2=36 $ and $a^2=\frac{243}5$ correspond to double complex roots, while
$a^2=64 $ and $a^2=56$ correspond to double real roots.

For $|a|>8$ the sign of $r_a$ in the roots of $p(t)=0$ is negative, while for  $|a|<8$ this sign is negative.

Therefore all singular points of $X$ different  from $(0,0)$ are hyperbolic saddles.
The point
 $(0,0)$ is a hyperbolic node for $\sqrt{56}<|a|<8$ and hyperbolic saddle when
$|a|<\sqrt{56}$ or $|a|>8$.
\end{proof}

\begin{proposition} \label{prop:res} 
Consider the planar blowing-up $\vartheta(u,v)=(r^2\sin \theta, r\cos \theta)$ around the origin.
Then, in $(\theta, r)$-coordinates, the resolution of the axial configuration is as shown in 
Fig. \ref{fig:meio1}  and Fig. \ref{fig:zero}.
\end{proposition}

 \begin{figure}[ht]
\begin{center}
\fbox{\includegraphics[scale=0.40]{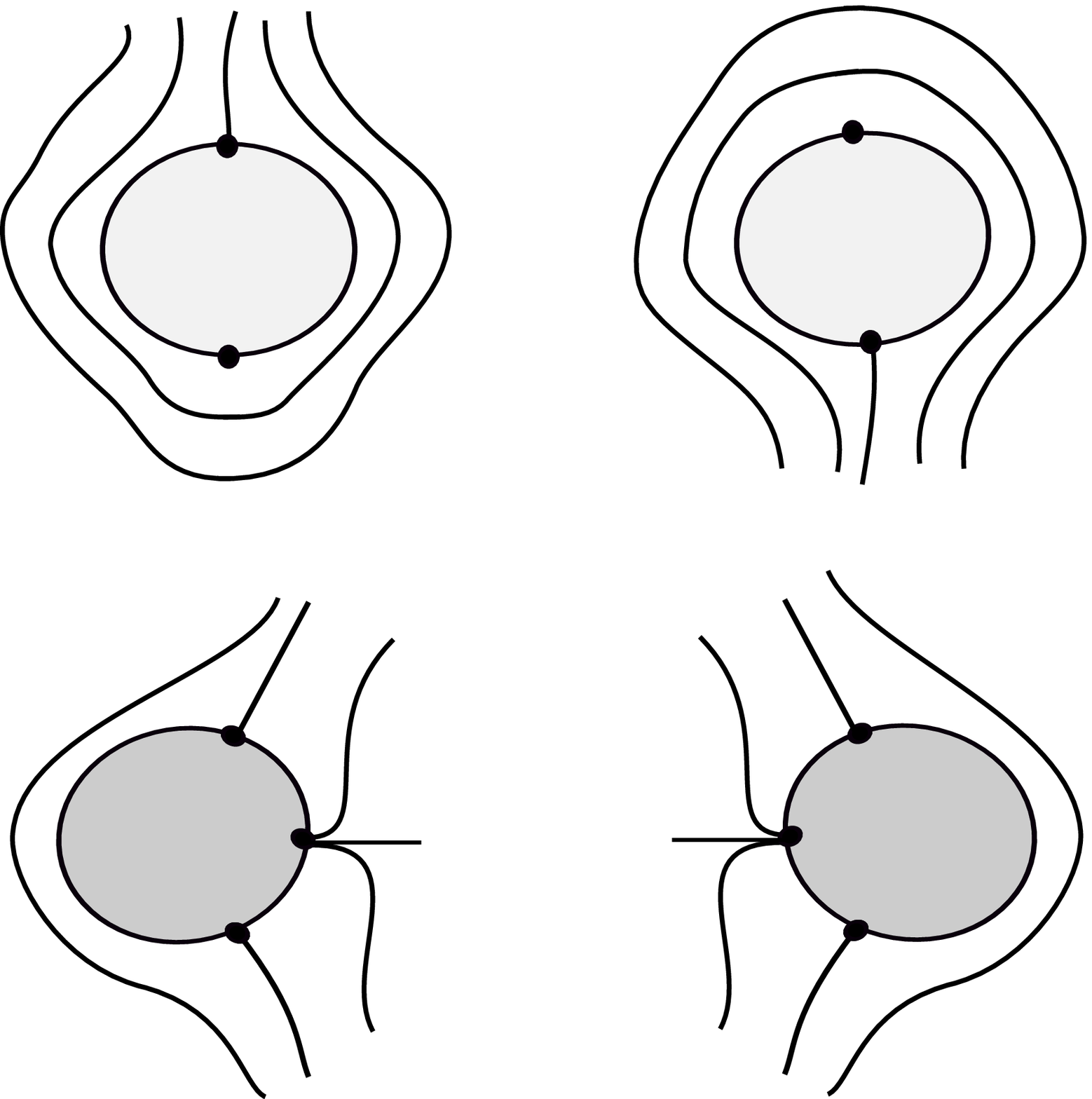}}
 \fbox{\includegraphics[scale=0.40]{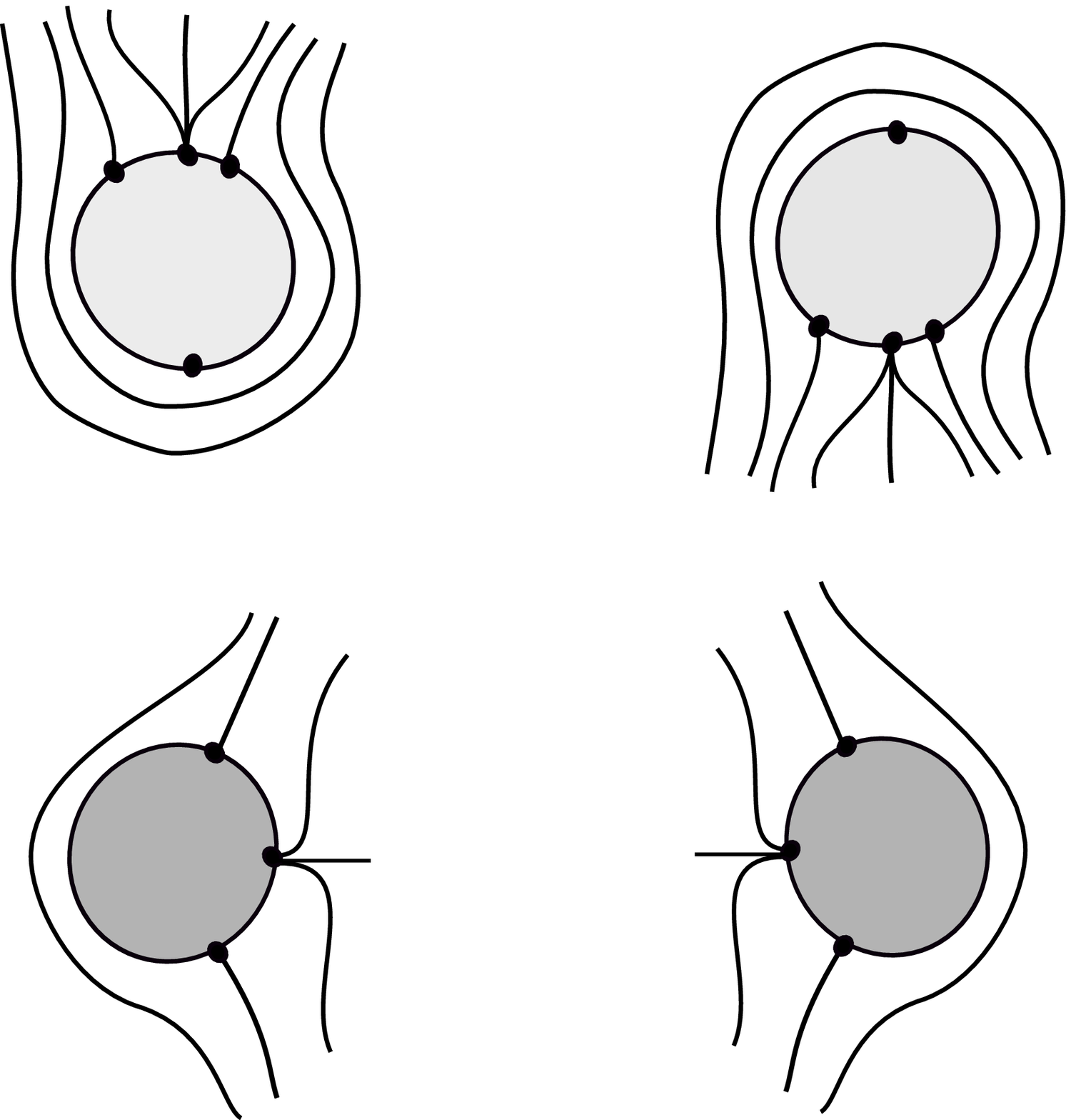}}
\caption{ Singular   points of index $1/2$. Left, with
  $|a|<\sqrt{56}$ and Right, with $8>|a|>\sqrt{56}.$ \label{fig:meio1}}
\end{center}
\end{figure}

 \begin{figure}[ht]
\begin{center}
\fbox{\includegraphics[scale=0.45]{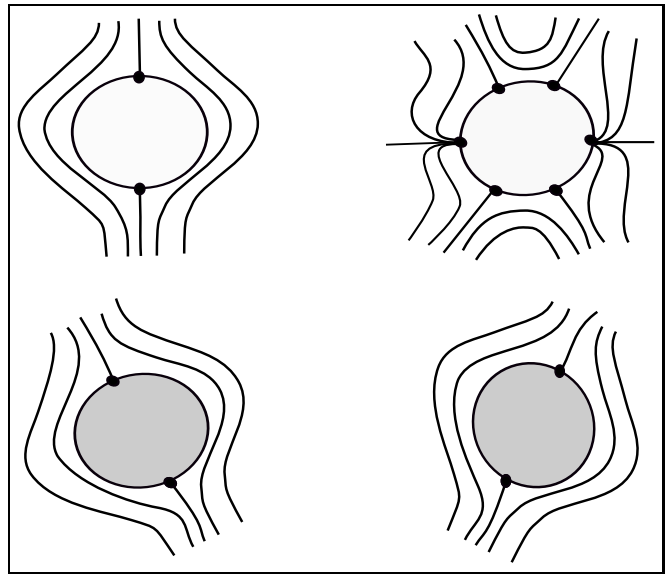}}
{\caption{ Singular   points of index $0$, $|a|>8$.  Behavior of the axial configuration in each topological disk.  \label{fig:zero}  }}
\end{center}
\end{figure}
 
\begin{proof} For $a=0$ the map $\alpha(u,v)=(u,uv,v^2,0)$ is the Whitney stable map.
In this case the axial configuration is given by the principal lines and the mean curvature lines. See  \cite{ggs}, \cite{sotogarcia1}, \cite{sotogarcia}. The Lie-Cartan 
variety  \eqref{eq:23sing}
is a pair of cylinders intersecting along the projective line. See Fig. \ref{fig:cm}. 
 The resolution is as shown in Fig. \ref{fig:meio1}, left.
By continuation, for $|a|<\sqrt{56}$ there are no bifurcation in the resolution. The induced differential equation has  three hyperbolic saddles in the interval  $(-\frac{\pi}2,\frac{\pi}2) $ and  other three hyperbolic saddles in the interval $( \frac{\pi}2,\frac{3\pi}2)$. The points $(\pm \frac{\pi}2,0)$ are hyperbolic nodes. 
 
For $\sqrt{56}<|a|<8$ the Lie-Cartan surface is still a pair of cylinders, but the induced differential equation   
has  four hyperbolic saddles in the interval $(-\frac{\pi}2,\frac{\pi}2)$ and 
$\theta=0$ is a hyperbolic node.  Also it has four hyperbolic saddles in the interval $(\frac{\pi}2,\frac{3\pi}2)$ and the points $(\pm \frac{\pi}2,0)$ are hyperbolic nodes.  See Fig. \ref{fig:meio1}, right.

For $|a|>8$ the Lie-Cartan 
variety 
is the union of four topological disks, see Fig. \ref{fig:4discos} and the induced differential equation   
has  five hyperbolic saddles in the interval $(-\frac{\pi}2,\frac{\pi}2)  $  and five hyperbolic saddles in the interval 
$(\frac{\pi}2,\frac{3\pi}2)$.  The points $(\pm \frac{\pi}2,0)$ are hyperbolic nodes.  See Fig. \ref{fig:zero}.
\end{proof}

 \begin{proposition} \label{prop:aco}
 Consider the map $\alpha(u,v)=(u,uv,v^2,\frac 16 a v^3)$ which has a Whitney singularity at $(0,0)$.

For $|a|<8$ the axial configuration has index $1/2$ at $(0,0)$ and when $|a|<\sqrt{56}$ the axial configuration is as shown in   of Fig. \ref{fig:acm} (left)   and for $  \sqrt{56}< |a|<8$ the axial configuration is as shown in  Fig. \ref{fig:acm}, (right).

 \begin{figure}[ht]
\begin{center}
\fbox{
\includegraphics[scale=0.40]{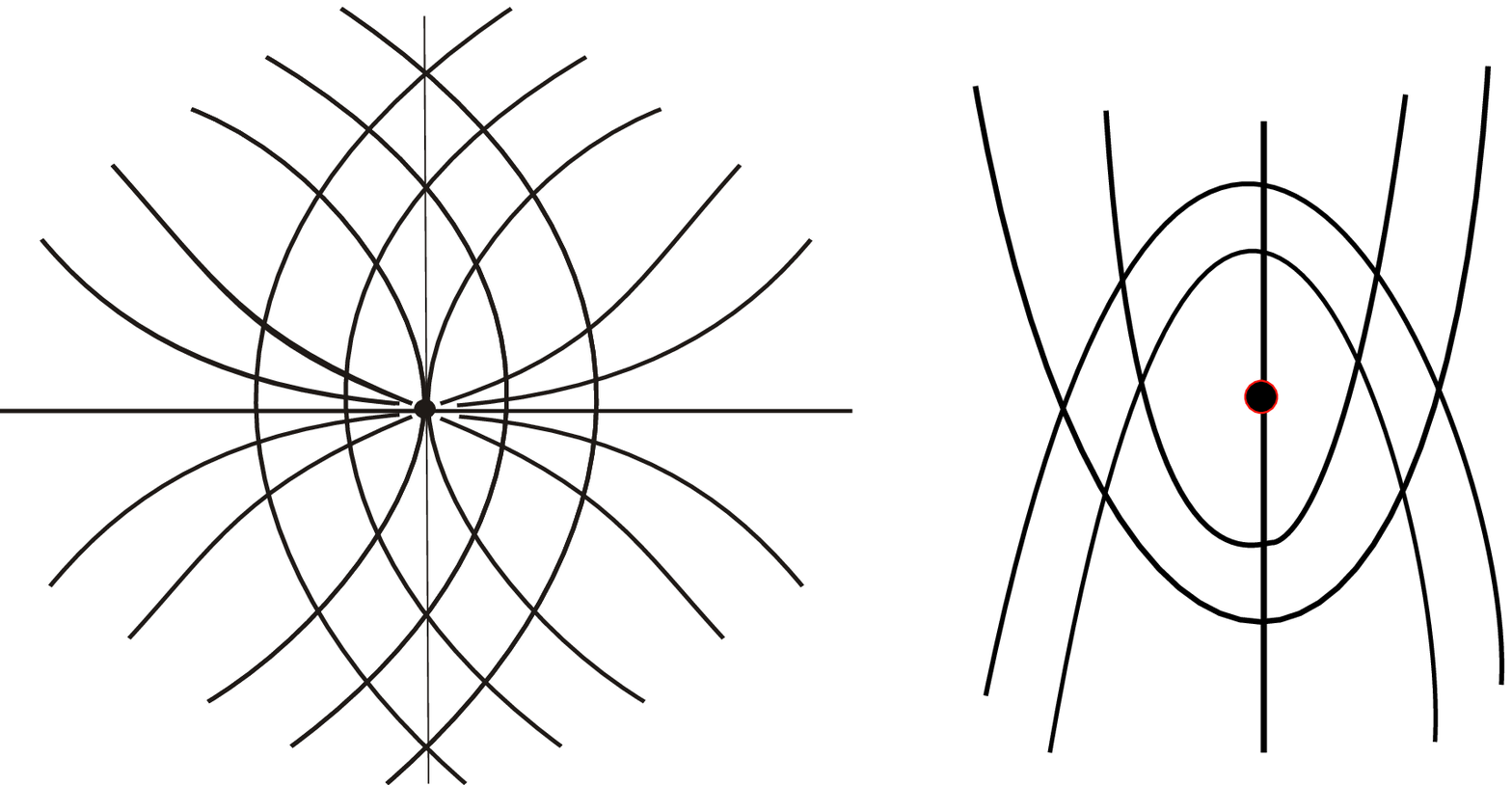}}
\fbox{\includegraphics[scale=0.40]{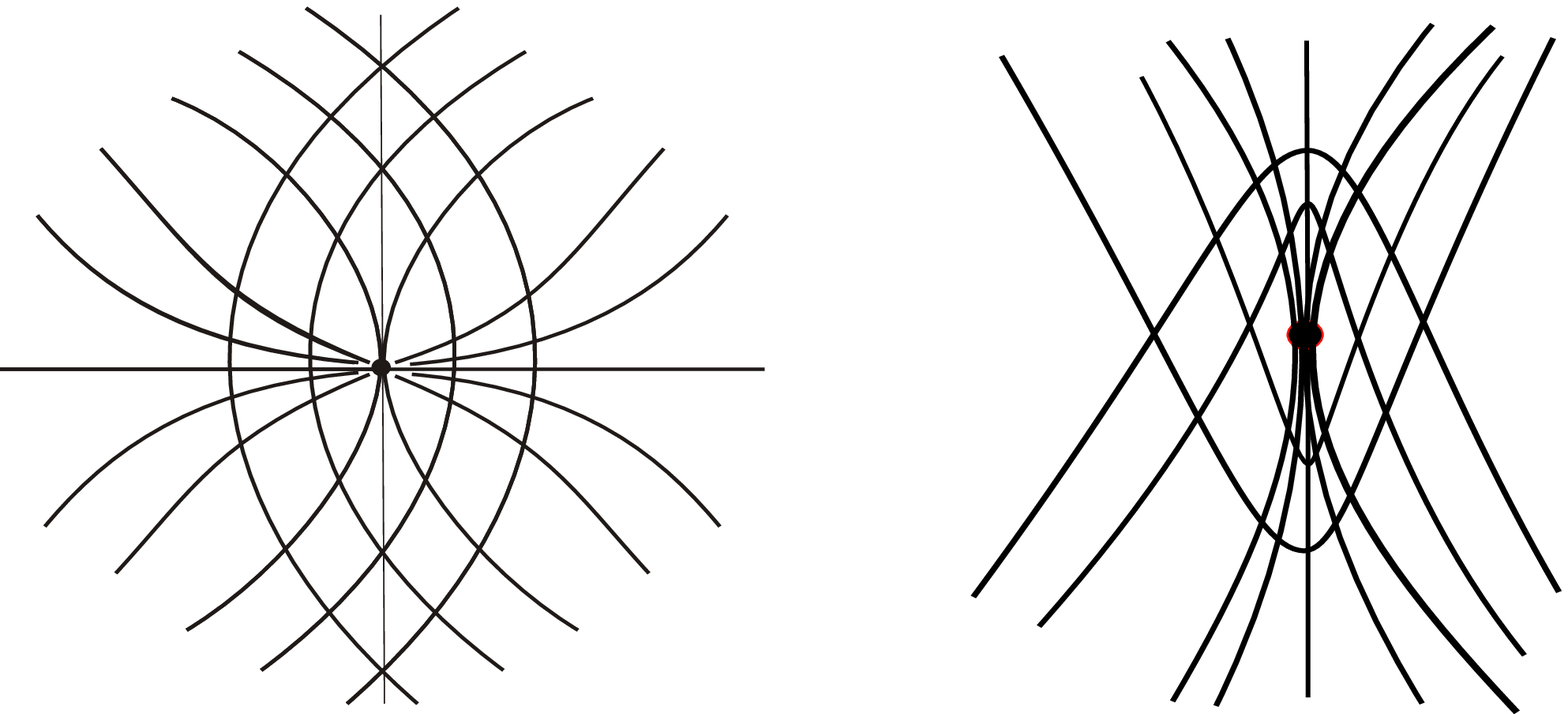}}
\caption{ Axial Configurations near a critical  point of index $\frac 12$.
Left for $|a|<\sqrt{56}$    and right for $\sqrt{56}< |a|< 8$.\label{fig:acm} }
\end{center}
\end{figure}

For $|a|>8$ the axial configuration has index $0$ at $(0,0)$ and it is as shown in Fig. \ref{fig:aco}.
 \begin{figure}[ht]
\begin{center}
\fbox{
\includegraphics[scale=0.55]{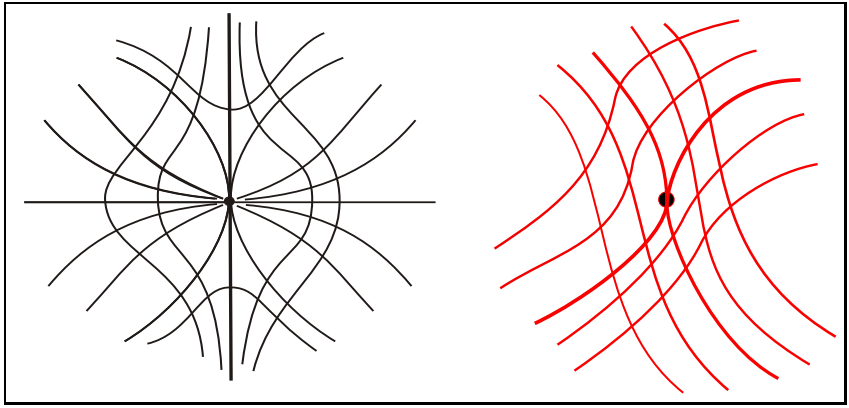}}
\caption{ Axial Configurations near a singular point of index $0$  ( $|a|>8$).\label{fig:aco} }
\end{center}
\end{figure}

\end{proposition}

\begin{proof} Follows from Proposition \ref{prop:res} 
performing 
the blowing-down of the resolution of the axial lines shown in Figs. \ref{fig:meio1} and \ref{fig:zero}.
\end{proof}

\begin{proposition} \label{prop:ace} For $\epsilon\ne 0$ small, consider the immersion $\alpha_\epsilon(u,v)=(u,uv,v^2,\epsilon v+\frac a6 v^3)$.
Then it follows that:

\noindent $\bullet$ For $|a|\leq 8$ and $\epsilon\ne 0$ the immersion $\alpha_\epsilon$    has two axiumbilic points.

\noindent $\bullet$   For $a>8$ the   immersion $\alpha_\epsilon$  has four axiumbilic points when $\epsilon >0$ and no axiumbilic points when $\epsilon <0$.

\noindent $\bullet$  For $a<-8$ the   immersion $\alpha_\epsilon$  has four axiumbilic points when $\epsilon <0$ and no axiumbilic points when $\epsilon >0$. 

Moreover, the axial configuration is as described below.

\noindent $\bullet$
For $a>8$, two axiumbilic points are of type $E_3$ and two are of type
 $E_5$.
  See Fig. \ref{fig:ace35}, top.

\noindent $\bullet$  For $a<-8$,  two axiumbilic points are of type $E_3$ and two are of type
 $E_5$. See Fig. \ref{fig:ace35}, bottom.
  
\noindent $\bullet$
For $ |a| < \frac{15}2$ the two axiumbilic points  are of type $E_3$ and the axial configuration is as  shown  in Fig. \ref{fig:ace}, left.

 \noindent $\bullet$ For   $\frac{15}2<|a| <8$ the two axiumbilic points are of type $E_4$ and the axial configuration is as  illustrated  in Fig. \ref{fig:ace}, right.

\begin{figure}[ht]
\begin{center}
\fbox{\includegraphics[scale=0.57]{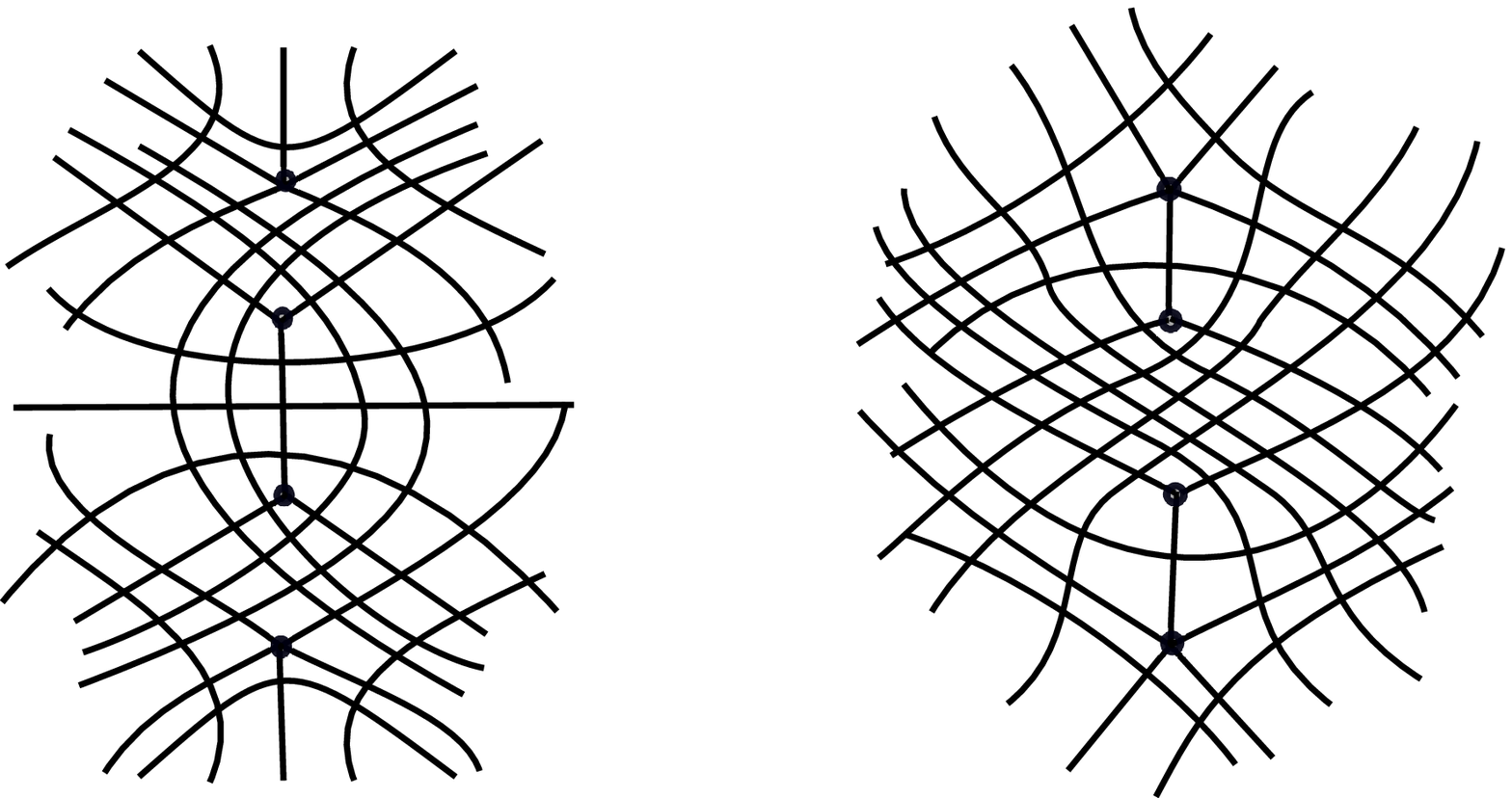}}
 \fbox{\includegraphics[scale=0.59]{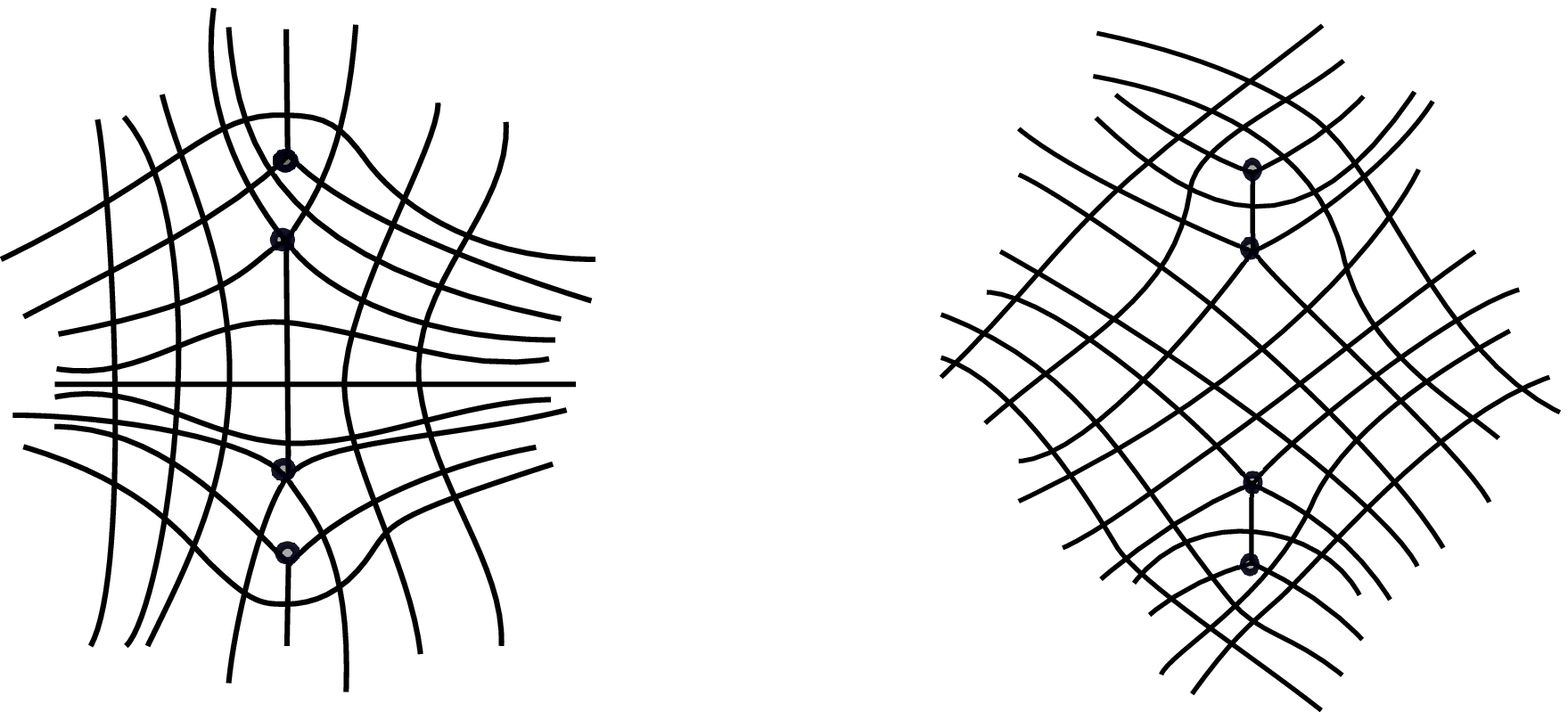}}
\caption{ Axial Configurations near axiumbilic points $E_3$ and $E_5$, branching from index $0$  critical point. Top, $a>8$ and bottom, $a<-8$. \label{fig:ace35} }
\end{center}
\end{figure}

 \begin{figure}[ht]
\begin{center}
\fbox{\includegraphics[scale=0.37]{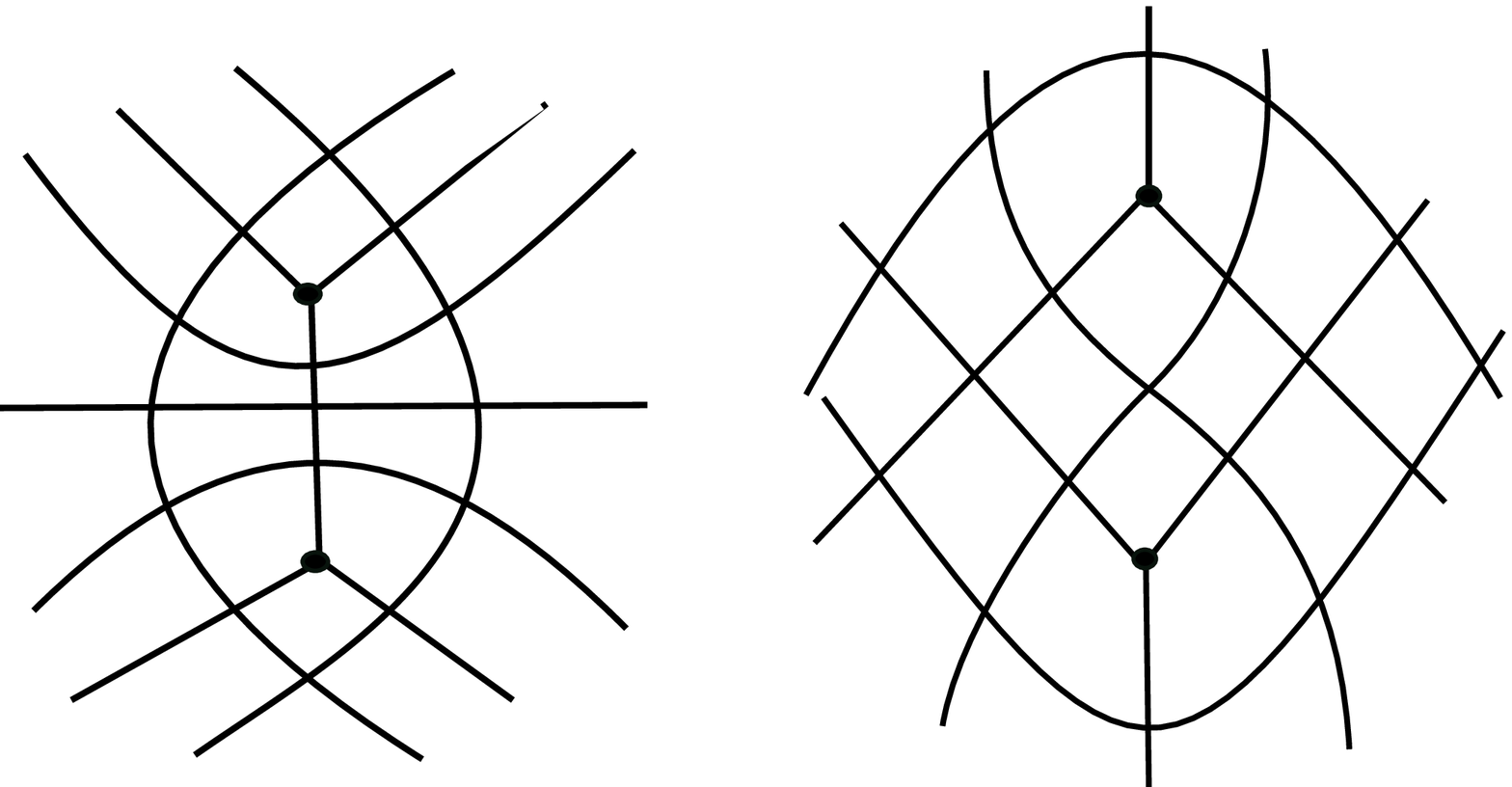}}
\fbox{\includegraphics[scale=0.36]{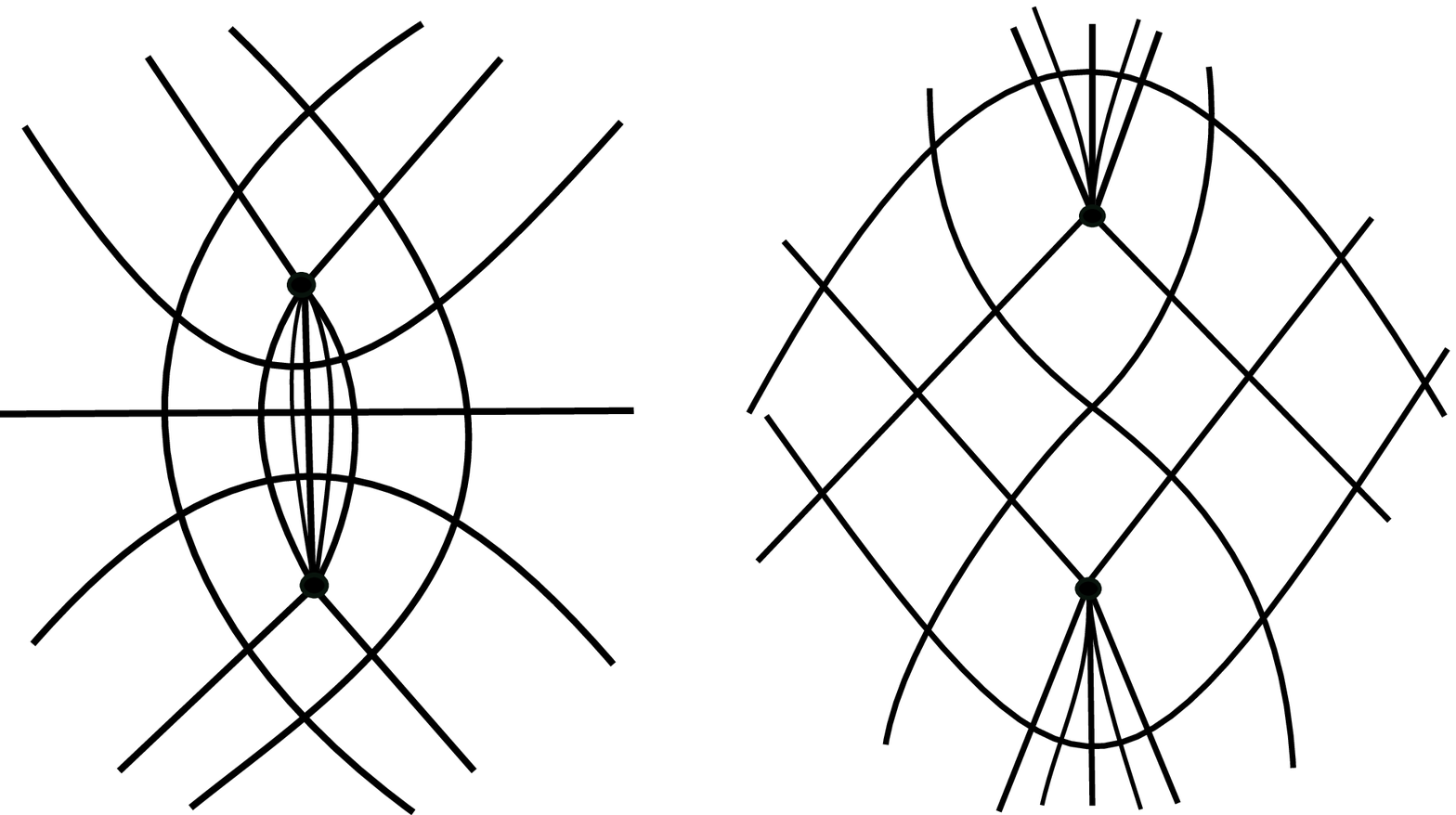}}
\caption{ Axial Configurations near axiumbilic points of types $E_3$, branching from the  index $1/2$ critical point.  Left for $|a| <\frac{15}2$  and
of types $E_4$, right for $\frac{15}2<|a|<8$. \label{fig:ace} }
\end{center}
\end{figure}

 \end{proposition}

 \begin{proof}
As in equation \eqref{eq:23sing}   it follows that
  
$$\begin{array}{lcr}
  \overline{e_1} = 0,  \;\;\;
 \overline{f_1}=4v+\frac 12 a^2 v^3+a\epsilon v,\;\;
\overline{g_1}=  -u(4+a^2v^2)\\
\overline{e_2}=0,\;\;\;
 \overline{f_2}= auv^2-2\epsilon u,\;\;
\overline{g_2}= \frac 12 av^3(1+v^2)(8+a^2v^2)-(8\epsilon +2a\epsilon^2)(1+v^2)
 \end{array}
$$
 Therefore,

\begin{equation}\label{eq:a0a1be} 
\aligned \overline{a_1}=&[(24-2a^2)v^2-8] u^2+\frac 12  v^4[(a^2-2a)v^2 +16-2a][(a^2+2a)v^2+16+2a]
 \\
 +& 8\epsilon^4+16a\epsilon^3v^2+(8 u^2-8v^4+12a^2 v^4-8+48v^2)\epsilon^2\\
 +&4av^2(2+a^2v^4+2u^2+20v^2+2v^4)\epsilon \\
 \overline{a_0}=&uv[8+24v^2+a^2v^2+2a^2v^4+2a\epsilon(1+3v^2)+4\epsilon^2].
 \endaligned 
 \end{equation}

 The axiumbilic points of $\alpha_\epsilon$ are defined by $\overline{a_0}=\overline{a_1}=0$.

 So, for $\epsilon $ small they are defined by:

 $$ u^2+\epsilon^2=0, \;\; v=0$$
 $$\aligned u&=0,\;\;[(2a+a^2)v^4+(-4\epsilon+4a \epsilon  +16+2a)v^2+4\epsilon^2-4\epsilon].\\
 [&(-2a+a^2)v^4+(4\epsilon+ 4a\epsilon+16-2a)v^2+4\epsilon^2+4\epsilon]=0\endaligned$$

 For $\epsilon=0$ it follows that $ \frac 12 v^4(v^2a^2-2a+16-2av^2)(v^2a^2+2a+16+2av^2)$.

 Solving the equation

 $$ [(2a+a^2)v^4+(-4\epsilon+4a \epsilon  +16+2a)v^2+4\epsilon^2-4\epsilon]=0$$  
with respect   to $ \epsilon$  it follows that

 $$\aligned \epsilon=&  \frac 12  +\frac 12 (1 +a)v^2+\frac 12\sqrt{1+(1-4a)v^4-(4a+14)v^2}=& \\
  =& \frac 12(8+a)v^2+(12+8a+a^2)v^4+O(5)=\epsilon_1(v)
  \endaligned$$

 Solving the equation
 $$(-2a+a^2)v^4+(4\epsilon+ 4a\epsilon+16-2a)v^2+4\epsilon^2+4\epsilon]=0$$
\noindent  it follows that
   $$\aligned \epsilon=&  -\frac 12  -\frac 12(1 +a)v^2+\frac 12\sqrt{ 1+(4a+1)v^4+(4a-14)v^2}=& \\
  =&  \frac 12 (a-8)v^2-(12+a^2-8a)v^4+O(5)=\epsilon_2(v).\endaligned$$

  For $ a^2-64\leq 0$ the two curves $(v,\epsilon_1(v))$ and $(v,\epsilon_2(v))$
are tangent at $0$ and have  even contact 
of 
opposite signs.

 So, for
$\epsilon\ne 0$ there are two axiumbilic points $(0,\pm v_0(\epsilon))$ symmetric with respect to the $u$-axis. See Fig. \ref{fig:ax2}, center.

\begin{figure}[ht]
\begin{center}
\fbox{\includegraphics[scale=0.5]{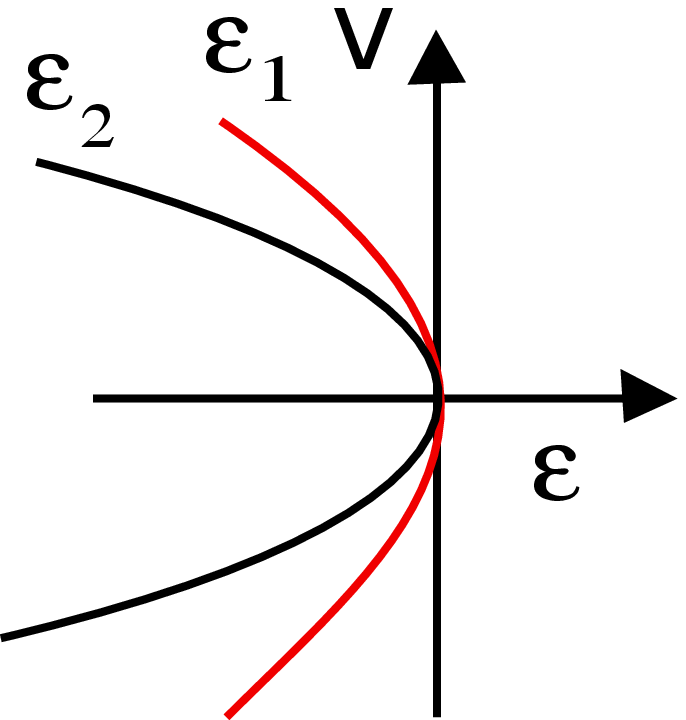}}\;\;\;\;\hskip 1cm
\fbox{\includegraphics[scale=0.5]{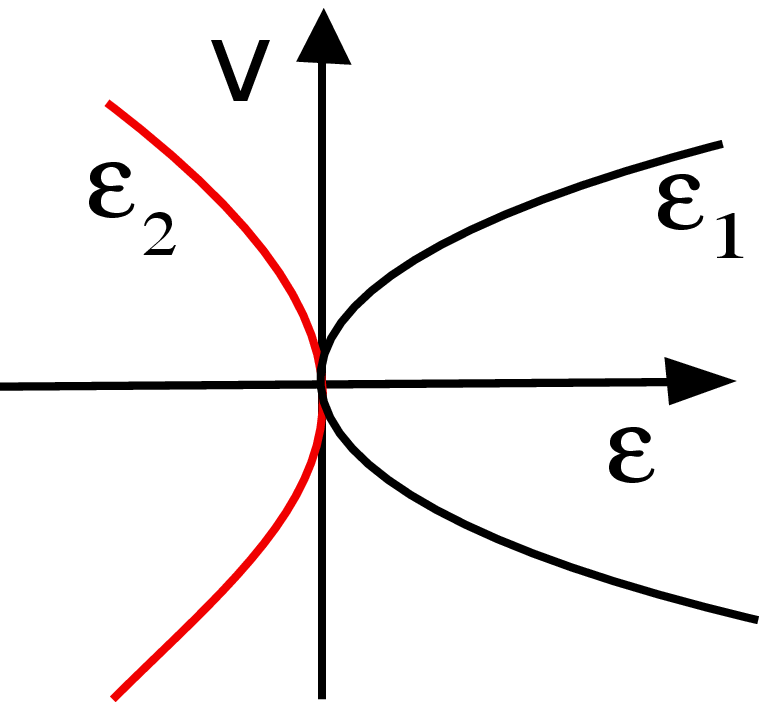}}\;\;\;\;\hskip 1cm
\fbox{\includegraphics[scale=0.5]{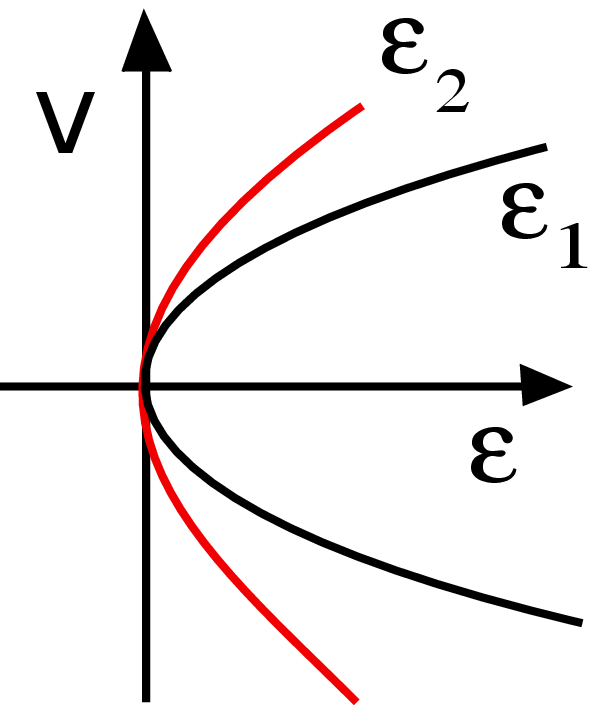}}
\caption{ Definition of  axiumbilic points \label{fig:ax2} (left, $a<-8$), (center, $|a|<8$),
(right, $a>8$).}
\end{center}
\end{figure}

For $ a^2-64> 0$ the two curves $(v,\epsilon_1(v))$ and $(v,\epsilon_2(v))$
are tangent at $0$ and have  even contact 
of 
the 
same signs. See Fig. \ref{fig:ax2} (left and right). Therefore $\alpha_\epsilon$ has four axiumbilic points when $a>8$ and $\epsilon >0$ and  
also when
$a<-8$ and $\epsilon <0$.

\vskip .4cm

\noindent{\bf Analysis of the axial configuration}.

 Consider the map $\beta(u,v)=(\overline{a_0}(u,v),\overline{a_1}(u,v))$, where $\overline{a_0}$ and $\overline{a_1}$ are given by equation \eqref{eq:a0a1be}.
 
 At an axiumbilic point defined by $\epsilon_1(v_0)=\epsilon$ it follows that
 $\text{det}D\beta(0,v_0)=512 (8+a)v_0^4+O(5)$ and for an axiumbilic point defined by
  $\epsilon_2(v_0)=\epsilon$ it follows that
 $\text{det}D\beta(0,v_0)=-512 (a-8)v_0^4+O(5)$.
 
 An axiumbilic point  is of type $E_5$ when $\text{det}D\beta(0,v_0)<0$, see \cite{sotogarcia1}.
 When $\text{det}D\beta(0,v_0)>0$ the axiumbilic point is of index $1/4$ and the type is not characterized by this sign.
 This  follows from Proposition \ref{prop.axial} and \ref{th:41}.

 Therefore, for $a>8$  two axiumbilic points defined by $\epsilon_2(v_0)=\epsilon$  are of type $E_5$ (index $-\frac 14$) and for $a<-8$ the two axiumbilic points defined by $\epsilon_1(v_0)=\epsilon$ are of type $E_5$.

For $|a|>8$   
it will 
be 
shown below that they are of type $E_3$.

To describe the type of these axiumbilic points  consider
 the linearization of the differential equation \eqref{eq:dasing} of axial lines at an axiumbilic point $(0,v_0(\epsilon))$.

Performing 
the calculations it follows that for  $\epsilon=\epsilon_1(v_0)$,

\begin{equation}\aligned 
{\mathcal E}_1(u,v,du,dv)=&A_4dv^4+A_3 dudv^3+A_2 dv^2 du^2+A_1dvdu^3+A_0du^4=0\\
A_0(u,v)= & [8v_0+O(2)] u+ 0.(v-v_0)+\ldots\\
 A_1(u,v)=& 0. u+[ 64(8+a)v_0^3+O(4)](v-v_0)+\ldots\\
 A_2(u,v)=& [-192v_0^3+0(4)]u+ 0.(v-v_0)+\ldots\\
A_3(u,v)=& 0.u - [256(a+8) v_0^5+O(6)].(v-v_0)+\ldots\\
 A_4(u,v)=& [128 v_0^5+0(6)]u+ 0.(v-v_0)+\ldots\endaligned\end{equation}

Further  calculations  for  $\epsilon=\epsilon_2(v_0)$ lead to
\begin{equation}\aligned 
{\mathcal E}_2(u,v,du,dv)=& A_4dv^4+A_3 dudv^3+A_2 dv^2 du^2+A_1dvdu^3+A_0du^4=0\\
A_0(u,v)= & [8v_0+O(2)] u+ 0.(v-v_0)+\ldots\\
A_1(u,v)=& 0. u+[ 64(8-a)v_0^3+O(4)](v-v_0)+\ldots\\
 A_2(u,v)=& [-192v_0^3+0(4)]u+ 0.(v-v_0)+\ldots\\
 A_3(u,v)=& 0.u - [256(8-a) v_0^5+O(6)].(v-v_0)+\ldots\\
A_4(u,v)=& [128 v_0^5+0(6)]u+ 0.(v-v_0)+\ldots\endaligned\end{equation}

 To  determine  the   type    of 
 the axiumbilic points $(0,v_0)$ defined by $\epsilon=\epsilon_1(v_0)$ consider the linear differential equation

 $$\aligned E_1(u,v,du,dv)=& 8v_0 du^4+  64(8+a)v_0^3(v-v_0)du^3dv-192v_0^3 u du^2 dv^2\\ -&256(a+8) v_0^5.(v-v_0)dudv^3+
 128 v_0^5 u dv^4=0.\endaligned$$

 The separatrices of the linear equation are defined by $u=kv$
 where $k$ is a root of the polynomial

 $$p_1(k)=k[k^4+8v_0^2(5+a)k^2-16v_0^4(15+2a)]=0.$$

 For $a<-\frac{15}{2}$ the polynomial $p_1(k)$ has five real roots
  while 
   for $a>-\frac{15}{2}$
 it has only three real roots.

 Analogously, for the axiumbilic points $(0,v_0)$ defined by $\epsilon=\epsilon_2(v_0)$
 consider the differential equation

 $$\aligned E_2(u,v,du,dv)=& 8v_0 du^4+  64(8-a)v_0^3(v-v_0)du^3dv-192v_0^3 u du^2 dv^2\\ -&256(8-a) v_0^5.(v-v_0)dudv^3+
 128 v_0^5 u dv^4=0.\endaligned$$

 The separatrices are defined by $u=kv$ where
 $$p_2(k)= k[k^4+ 8v_0^2(a-5)k^2+16v_0^4(2a-15)]=0.$$

 For $a>\frac{15}{2}$ the polynomial $p_2(k)$ has five real roots, while for $a<\frac{15}{2}$
 it has only three real roots.

 Therefore, by the classification of axiumbilic points, 
 see \cite{sotogarcia1},
  for
 $|a|<\frac{15}2$ and $\epsilon >0$ the axiumbilic points $(0,v_0)$ are defined by $\epsilon=\epsilon_1(v_0)$ and they are all of type $E_3$. 
Also  $\frac{15}2<|a|<8$ and $\epsilon >0$ all axiumbilic points are defined by $\epsilon=\epsilon_1(v_0)$, see Fig. \ref{fig:ax2} center, 
 and they are of type $E_4$. Analogously, when $\epsilon <0$ the axiumbilic points are defined by $\epsilon=\epsilon_2(v_0)$ and the same analysis and conclusions can be established.

For  $a>8$ and $\epsilon>0$ the axiumbilic points defined by $\epsilon=\epsilon_2(v_0)$
 are of type $E_5$ and
 the other two defined by $\epsilon=\epsilon_1(v_0)$ have the separatrices given by $p_1(k)=0$ which has three real roots for $a>-\frac{15}2$.  Therefore they are of type $E_3$, see Fig. \ref{fig:ace35}, top.

 For  $a<-8$  and $\epsilon<0$  the axiumbilic points defined   
   by $\epsilon=\epsilon_1(v_0)$ are of type $E_5$ and the other two defined by $\epsilon=\epsilon_2(v_0)$
 have the separatrices given by $p_2(k)=0$ which has three real roots for $a<\frac{15}2$.  Therefore they are of type $E_3$, see Fig. \ref{fig:ace35}, bottom.

    The analysis above also follows from Proposition \ref{prop.axial}, taking into account Fig. \ref{fig:d345}, illustrating Theorem \ref{th:41}  established in  \cite{sotogarcia1}.   \end{proof}

\section{Concluding  Comments} \label{sec:f_Comments}

The results established in this work are motivated and provide a continuation 
of the previous paper by Garcia,
Sotomayor  and  Spindola  \cite{flausino_arxiv}.

The authors believe that these  results   describe  
 the axial configurations at the generic  critical point of a surface mapped into $\R ^4$, as illustrated in Figures \ref{fig:acm} and \ref{fig:aco}.  The results also  present  a  rough description of  important  partial elements of  the transversal,  codimension$-1$,  bifurcations occurring by the elimination of the critical point.  
 
For the full description  of the generic bifurcation phenomenon, a delicate  analysis  of the breaking of the axiumbilic separatrix connections in Figures  
\ref{fig:ace35}
and 
\ref{fig:ace}  
must be carried out.

 This connection breaking is due to the presence of coefficients of the  third order jet of the mapping omitted in the example treated here.

\vskip 1.5cm

\author{\noindent Ronaldo Garcia\\Instituto de Matem\'atica e Estat\'{\i}stica \\
Universidade Federal de Goi\'as,\\ CEP 74001--970, Caixa Postal 131 \\
Goi\^ania, Goi\'as, Brazil \\
 \email{ragarcia@mat.ufg.br  }
\vskip 1.0cm

\author{\noindent Jorge Sotomayor\\ Instituto de Matem\'atica e Estat\'{\i}stica \\
Universidade  de S\~ao Paulo,\\
 Rua do Mat\~ao  1010,
Cidade Univerit\'aria, CEP 05508-090,\\
S\~ao Paulo, S. P, Brazil \\
 \email{sotp@ime.usp.br}
 \vskip 0.5cm

\end{document}